\documentclass[11pt]{article}
\usepackage{graphicx}
 \usepackage{mathptmx,color}      
\usepackage[titletoc,title]{appendix}
\usepackage[numbers]{natbib}
\usepackage[T1]{fontenc}
\usepackage[applemac]{inputenc}
\usepackage[all]{xy}
\usepackage{enumerate}
\usepackage{comment}
\usepackage{amsmath,amssymb,stmaryrd}
\usepackage[british]{babel}
\usepackage{a4wide}
\usepackage[sans]{dsfont}
\usepackage{amsfonts}
\usepackage{amsthm}
\usepackage[mathscr]{euscript}
\usepackage{graphicx}
\usepackage{mathrsfs}
\usepackage{latexsym}
\usepackage{setspace}
\usepackage{amsthm,amsmath}

\newcommand{\p}{^}

\newcommand{\wt}{\widetilde}
\newcommand{\wh}{\widehat}

	\newcommand{\rmd}{\mathrm{d}}

\newcommand{\Lrm}{\mathrm{L}}

\newcommand{\Var}{\mathrm{Var}}

\newcommand{\Ascr}{\mathscr{A}}
\newcommand{\Bscr}{\mathscr{B}}
\newcommand{\Cscr}{\mathscr{C}}
\newcommand{\Dscr}{\mathscr{D}}
\newcommand{\Escr}{\mathscr{E}}
\newcommand{\Fscr}{\mathscr{F}}
\newcommand{\Gscr}{\mathscr{G}}
\newcommand{\Hscr}{\mathscr{H}}

\newcommand{\Kscr}{\mathscr{K}}

\newcommand{\Oscr}{\mathscr{O}}
\newcommand{\Pscr}{\mathscr{P}}

\newcommand{\Rscr}{\mathscr{R}}

\newcommand{\Mloc}{\Hscr\p1_{\mathrm{loc}}}
\newcommand{\Hloc}{\Hscr_{\mathrm{loc}}}
\newcommand{\Gloc}{\Gscr_{\mathrm{loc}}}

\newcommand{\et}{\eta}

\newcommand{\al}{\alpha}
\newcommand{\bt}{\beta}

\newcommand{\sig}{\sigma}

\newcommand{\om}{\omega}
\newcommand{\Om}{\Omega}
\newcommand{\cadlag}{c\`adl\`ag\ }

\newcommand{\Bb}{\mathbb{B}}
\newcommand{\Dbb}{\mathbb{D}}
\newcommand{\Ebb}{\mathbb{E}}

\newcommand{\Fbb}{\mathbb{F}}

\newcommand{\Nbb}{\mathbb{N}}
\newcommand{\Pbb}{\mathbb{P}}
\newcommand{\Qbb}{\mathbb{Q}}
\newcommand{\Rbb}{\mathbb{R}}

\newcommand{\Hbb}{\mathbb{H}}
\newcommand{\Gbb}{\mathbb{G}}

\newcommand{\aPP}[2]{\ensuremath{\langle #1,#2 \rangle}}

\newtheorem{theorem}{Theorem}[section]
\newtheorem{lemma}[theorem]{Lemma}
\newtheorem{proposition}[theorem]{Proposition}
\newtheorem{corollary}[theorem]{Corollary}
\theoremstyle{definition}
\newtheorem{definition}[theorem]{Definition}
\newtheorem{examples}[theorem]{Examples}

\newtheorem{assumption}[theorem]{Assumption}

\newtheorem{remark}[theorem]{Remark}

\numberwithin{equation}{section}

\makeatletter
\def\timenow{\@tempcnta\time
\@tempcntb\@tempcnta
\divide\@tempcntb60
\ifnum10>\@tempcntb0\fi\number\@tempcntb
:\multiply\@tempcntb60
\advance\@tempcnta-\@tempcntb
\ifnum10>\@tempcnta0\fi\number\@tempcnta}
\makeatother

\title{On the Propagation of the Weak Representation Property in Independently Enlarged Filtrations: The General Case.}
\author{Paolo Di Tella\\ {\emph{E-Mail: }{\tt Paolo.Di\_Tella{\rm@}tu-dresden.de}}
     }
\date{}

\begin{document}
\maketitle

\begin{abstract}
In this paper we investigate the propagation of the weak representation property (WRP) to an independently enlarged filtration. More precisely, we consider an $\Fbb$-semimartingale $X$ possessing the WRP with respect to $\Fbb$ and an $\Hbb$-semimartingale $Y$ possessing the WRP with respect to $\Hbb$. Assuming that $\Fbb$ and $\Hbb$ are independent, we show that the $\Gbb$-semimartingale $Z=(X,Y)$ has the WRP with respect to $\Gbb$, where $\Gbb:=\Fbb\vee\Hbb$. In our setting, $X$ and $Y$ may have simultaneous jump-times. Furthermore, their jumps may charge predictable times. This generalizes all available results about the propagation of the WRP to independently enlarged filtrations.
\end{abstract}

{\noindent \textit{Keywords:}  Weak representation property, semimartingales, progressive enlargement of filtrations, independent semimartingales, random measures, stochastic integration.
}

{\noindent \textit{AMS Classification:  	60G44; 	60G57; 	60H05; 	60H30}}
\section{Introduction}  

Let $X$ be a $d$-dimensional semimartingale with respect to a right-continuous filtration $\Fbb$ and let $X\p c$ denote the continuous local martingale part of $X$. We say that $X$ possesses the \emph{weak representation property} (from now on WRP) with respect to $\Fbb$ if every $\Fbb$-local martingale can be represented as the sum of a stochastic integral with respect to $X\p c$ and a stochastic integral with respect to the compensated jump measure of $X$ (for details see Definition \ref{def:w.prp} below). 

The WRP of $X$ is a property depending on the filtration $\Fbb$: For example, if $X$ is a L\'evy process and $\Fbb=\Fbb\p X$ is the smallest right-continuous filtration with respect to which $X$ is adapted, then $X$ possesses the WRP with respect to $\Fbb\p X$. However, this need not be true if $X$ is considered with respect to a larger filtration. Therefore, it is natural to investigate under which conditions and in which form the WRP of a semimartingale $X$ with respect to a filtration $\Fbb$ propagates to a larger filtration $\Gbb$. 

In this paper we suppose that $\Fbb$ is enlarged by a right-continuous filtration $\Hbb$ and we denote by $\Gbb$ the smallest right-continuous filtration containing both $\Fbb$ and $\Hbb$. We assume that $\Hbb$ has the following properties: (1) $\Hbb$ is independent of $\Fbb$; (2) $\Hbb$ supports an $\Rbb\p\ell$-valued semimartingale $Y$ possessing the WRP with respect to $\Hbb$. Under these conditions, we show in Theorem \ref{thm:mar.rep.G} below (the main result of this paper) that the $\Rbb\p d\times\Rbb\p \ell$-valued $\Gbb$-semimartingale $Z=(X,Y)$ possesses the WRP with respect to $\Gbb$. We stress that, in the present paper, we do not make any further assumption: The semimartingales $X$ and $Y$ may have simultaneous jump-times or their jumps may charge predictable times. To the best of our knowledge, Theorem \ref{thm:mar.rep.G} below is the most general result about the propagation of the WRP to an independently enlarged filtration.

The propagation of the WRP to an independently enlarged filtration has been studied in Xue \cite{X93} under the further assumption that $X$ (or $Y$) are quasi-left continuous (i.e., the jumps do not charge predictable jump-times). At a first look, this assumption seems to be harmless and fairly general. On the other side, it leads to the following strong simplification of the problem: $\Fbb$-local martingales and $\Hbb$-local martingales have no common jumps. Contrarily, in the present paper we face the additional difficulty to determine an adequate representation of the simultaneous jumps of $\Fbb$- and $\Hbb$-local martingales. Moreover, examples of non-necessarily quasi-left continuous semimartingale possessing the WRP are known (see Example \ref{ex:exam} below) and, in this case, the propagation of the WRP to the independently enlarged filtration $\Gbb$ cannot be derived from \cite{X93}.
 
In Wu and Gang \cite{WG82}, under the independence assumption of $\Fbb$ and $\Hbb$, necessary and sufficient conditions on the semimartingale characteristics of $X$ and $Y$ are stated for the $\Gbb$-semimartingale $X+Y$ to possess the WRP with respect to $\Gbb$. In particular, the authors do not assume that $X$ (or $Y$) are quasi-left continuous, but only that the sets of their accessible jump-times are disjoint. This however again yields that $\Fbb$-local martingale and $\Hbb$-local martingale have no common jumps (see \cite[Lemma 7]{WG82}).

If $(X,\Fbb)$ and $(Y,\Hbb)$ are local martingales, a first work investigating martingale representation theorems in the independently enlarged filtration $\Gbb$, without quasi-left continuity assumptions and allowing simultaneous jumps of $X$ and $Y$, is Calzolari and Torti \cite{CT16}. However, \cite{CT16} deals with the propagation of the \emph{predictable representation property} (from now on PRP\footnote{We recall that a local martingale $X$ possesses the PRP with respect to a filtration $\Fbb$ if every $\Fbb$-local martingale is a stochastic integral with respect to $X$ of an $\Fbb$-predictable process.}) and not with the WRP. We recall that the WRP is a more general property than the PRP (see, e.g., \cite[Theorem 13.14]{HWY92} or Lemma \ref{prop:WRP.PRP} below). In Calzolari and Torti \cite{CT19}, the results of \cite{CT16} are extended to multidimensional local martingales.  We show in Corollary \ref{cor:srstr.rep} below, that the results obtained by \cite{CT16,CT19} can be reformulated in terms of the WRP of the $\Gbb$-local martingale $Z=(X,Y)$.

If the filtration $\Gbb$ is obtained enlarging the filtration $\Fbb$ by a non-necessarily independent filtration $\Hbb$, then very little is known about the propagation of the WRP. Results in this direction are available if $\Fbb$ is enlarged progressively by a random time $\tau$ that need not be an $\Fbb$-stopping time: In this case, $\Hbb$ is generated by the process $1_{[\tau,+\infty)}$ and, therefore, $\Gbb$ is the smallest right-continuous filtration containing $\Fbb$ and such that $\tau$ is a stopping time. We are now going to review these results.

In Barlow \cite{Ba78}, a semimartingale $X$ possessing the WRP with respect to a filtration $\Fbb$ is considered and a WRP is obtained in $\Gbb$ if $\tau$ is a \emph{honest time}. However, honest times are (morally) $\Fscr_\infty$-measurable random variables and, therefore, this excludes the independent enlargement. 

In Di Tella \cite{DT19}, the WRP in $\Gbb$ is obtained under the assumptions that $\Fbb$-martingales are $\Gbb$-martingales (that is, if the \emph{immersion property} holds) and $\Pbb[\tau=\sig<+\infty]=0$, for all $\Fbb$-stopping times $\sig$ (that is, if $\tau$ \emph{avoids} $\Fbb$ stopping times). If, from one side, the independent enlargement is a special case of the immersion property, on the other side, the avoidance of stopping times yields that $\tau$ cannot be charged by the jumps of $\Fbb$-local martingales. 

In summary, also if $\Fbb$ is progressively enlarged by a random time $\tau$, the case studied in the present paper cannot be covered by \cite{Ba78} nor by \cite{DT19}. 

The present work has the following structure: In Section \ref{sec:bas}, we recall some basic definitions ad results needed in this paper. In Section \ref{sec:prop.mar.rep}, we prove our main result (Theorem \ref{thm:mar.rep.G}) about the propagation of the WRP to the independently enlarged filtration $\Gbb$. Then, as a consequence of Theorem \ref{thm:mar.rep.G}, we also investigate the propagation of the PRP to $\Gbb$. Section \ref{sec:apl} is devoted to two applications of Theorem \ref{thm:mar.rep.G}: In Subsection \ref{subs:seq.en}, we study the propagation of the WRP to the \emph{iterated} independent enlargement. This part is inspired to \cite[\S 4.1]{CT19}. In Subsection \ref{subs:jeq}, we assume that the filtration $\Fbb$ is enlarged by a random time $\tau$ satisfying the so-called \emph{Jacod's equivalence hypothesis}. We stress that in Subsection \ref{subs:jeq} the filtrations $\Fbb$ and $\Hbb$ need not be independent. Therefore, Subsection \ref{subs:jeq} is an extension of Theorem \ref{thm:mar.rep.G}. We also observe that, in Subsection \ref{subs:jeq}, $\tau$ need not avoid $\Fbb$-stopping times. Hence, this part extends known results as Callegaro, Jeanblanc and Zargari \cite[Proposition 5.5]{CJZ13}, where the avoidance property is additionally assumed. Finally, we postpone to the Appendix the proof of some technical results.

\section{Basic Notions}\label{sec:bas}
In this paper we regard $d$-dimensional vectors as \emph{columns}, that is, if $v\in\Rbb\p d$, then $v=(v\p1,\ldots,v\p d)\p{tr}$, where $tr$ denotes the transposition operation and $v\p i\in\Rbb$. If $v\p i\in\Rbb\p {d_i}$, $d_i\geq1$, $i=1,\ldots,n$, we denote by $(v\p 1,v\p2,\ldots,v\p n)\p{tr}$ the $d_1+d_2+\ldots+d_n$-dimensional column-vector, obtained continuing with $v\p2$ after $v\p1$ and so on, till $v\p n$. 
  
\indent\textbf{Stochastic processes, filtrations and martingales.} 
Let $(\Om,\Fscr,\Pbb)$ be a complete probability space. 
For any \cadlag\  process $X$, we denote by $\Delta X$ the jump process of $X$, i.e., $\Delta X_t:=X_t-X_{t-}$, $t>0$, and $X_{0-}:=X_0$.

We denote by $\Fbb=(\Fscr_t)_{t\geq0}$ a right-continuous filtration and by $\Oscr(\Fbb)$ (resp.\ $\Pscr(\Fbb)$) the $\sig$-algebra of the $\Fbb$-optional (resp.\ $\Fbb$-predictable) sets of $\Om\times\Rbb_+$, $\Rbb_+:=[0,+\infty)$. We set $\Fscr_\infty:=\bigvee_{t\geq0}\Fscr_t$. We sometime use the notation $(X,\Fbb)$ to denote an $\Fbb$-adapted stochastic process $X$.

For a process $X$, we denote by $\Fbb\p X$ the smallest right-continuous filtration such that $X$ is adapted.

Let $X$ be an $[-\infty,+\infty]$-valued and $\Fscr\otimes\Bscr(\Rbb_+)$-measurable process, where $\Bscr(\Rbb)$ denotes the Borel $\sig$-algebra on $\Rbb$. We denote by ${}\p{p,\Fbb} X$ the (\emph{extended}) $\Fbb$-predictable projection of $X$. For the definition ${}\p{p,\Fbb} X$,
we refer to \cite[Theorem I.2.28]{JS00}.

An $\Fbb$-adapted \cadlag process $X$ is called \emph{quasi-left continuous} if $\Delta X_T=0$ for every finite-valued and $\Fbb$-predictable stopping time $T$.

For $q\geq1$, we denote by $\Hscr\p q(\Fbb)$ the space of $\Fbb$-uniformly integrable martingales $X$ such that $\|X\|_{\Hscr\p q}:=\Ebb[\sup_{t\geq0}|X_t|\p q]\p{1/q}<+\infty$. Recall that $(\Hscr\p q,\|\cdot\|_{\Hscr\p q})$ is a Banach space. For $X\in\Hscr\p2(\Fbb)$, we also introduce the equivalent norm $\|X\|_2:=\Ebb[X\p2_\infty]\p{1/2}$ and $(\Hscr\p 2,\|\cdot\|_{2})$ is a Hilbert space. 

For each $q\geq1$, the space $\Hscr\p q_\mathrm{loc}(\Fbb)$ is introduced from $\Hscr\p q(\Fbb)$ by localization. We observe that $\Hscr\p1_\mathrm{loc}(\Fbb)$ coincides with the space of \emph{all} $\Fbb$-local martingales (see \cite[Lemma 2.38]{J79}). We denote by $\Hscr\p p_0(\Fbb)$ (resp., $\Hscr\p q_{\mathrm{loc},0}(\Fbb)$) the subspace of martingales (resp., local martingales) $Z\in\Hscr\p q(\Fbb)$ (resp., $Z\in\Hscr\p p_\mathrm{loc}(\Fbb)$) such that $Z_0=0$. 

Two local martingales $X$ and $Y$ are called \emph{orthogonal} if $XY\in\Hscr\p1_{\mathrm{loc},0}(\Fbb)$ holds.
For $X,Y\in\Hscr\p2_\mathrm{loc}(\Fbb)$ we denote by $\aPP{X}{Y}$ the \emph{predictable covariation} of $X$ and $Y$. We recall that $XY-\aPP{X}{Y}\in\Hscr\p1_\mathrm{loc}(\Fbb)$. Hence, if $X_0Y_0=0$, then $X$ and $Y$ are orthogonal if and only if  $\aPP{X}{Y}=0$.

An $\Rbb$-valued $\Fbb$-adapted process $X$ such that $X_0=0$ is called \emph{increasing} if $X$ is \cadlag and the paths $t\mapsto X_t(\om)$ are non-decreasing, $\om\in\Om$. We denote by $\Ascr\p+=\Ascr\p+(\Fbb)$ the space of $\Fbb$-adapted integrable processes, that is, $\Ascr\p+$ is the space of increasing process $X$ such that $\Ebb[X_\infty]<+\infty$ (see \cite[I.3.6]{JS00}). We denote by $\Ascr\p+_\mathrm{loc}=\Ascr\p+_\mathrm{loc}(\Fbb)$ the localized version of $\Ascr\p+$.
For $X\in\Ascr\p+_\mathrm{loc}$, we denote by $X\p p\in\Ascr\p+_\mathrm{loc}$ the $\Fbb$-dual predictable projection of $X$ (see \cite[Theorem I.3.17]{JS00}). 

Let $(X,\Fbb)$ be an increasing process and let $K\geq0$ be an $\Fbb$-optional process. We denote
 by $K\cdot X=(K\cdot X_t)_{t\geq0}$ the process defined by the (Lebesgue--Stieltjes) integral of $K$ with respect to $X$, that is, $K\cdot
 X_t(\om):=\int_0\p t K_s(\om)\rmd X_s(\om)$, if $\int_0\p t K_s(\om)\rmd X_s(\om)$ is finite-valued, for every $\om\in\Om$ and $t\geq0$. Notice that $(K\cdot
 X,\Fbb)$ is an increasing process.

\textbf{Random measures.} Let $\mu$ be a nonnegative random measure on $\Rbb_+\times E$ in the sense of \cite[Definition II.1.3]{JS00}, where $E$ coincides with $\Rbb\p d$ or with a Borel subset of $\Rbb\p d$. We stress that we assume $\mu(\om,\{0\}\times E)=0$ identically. 

We denote by $\Bscr(E)$ the Borel $\sig$-algebra on $E$ and set $\wt\Om:=\Om\times\Rbb_+\times E$. We then introduce the following $\sig$-fields on $\wt\Om$: $\wt\Oscr(\Fbb):=\Oscr(\Fbb)\otimes\Bscr(E)$ and $\wt\Pscr(\Fbb):=\Pscr(\Fbb)\otimes\Bscr(E)$.

Let $W$ be an $\wt\Oscr(\Fbb)$-measurable (resp.\ $\wt\Pscr(\Fbb)$-measurable) mapping from $\wt\Om$ into $\Rbb$. We say that $W$ is an $\Fbb$-optional (resp.\ $\Fbb$-predictable) function. Let $W$ be an $\Fbb$-optional function. As in \cite[II.1.5]{JS00}, we define
\[
 W\ast\mu(\om)_t:=\begin{cases} \displaystyle\int_{[0,t]\times E}W(\om,t,x)\mu(\om,\rmd t,\rmd x),&\quad \textnormal{if } \displaystyle\int_{[0,t]\times E}|W(\om,t,x)|\mu(\om,\rmd t,\rmd x)<+\infty;\\\\
\displaystyle+\infty,&\quad\textnormal{else}.
\end{cases}
\]

We say that $\mu$ is an $\Fbb$-optional (resp.\ $\Fbb$-predictable) random measure if $W\ast\mu$ is an $\Fbb$-optional (resp.\ an $\Fbb$-predictable) process, for every optional (resp.\ $\Fbb$-predictable) function $W$.
\\[.1cm]\indent
\textbf{Semimartingales.} 
Let $X$ be an $\Rbb\p d$-valued $\Fbb$-semimartingale. We denote by $\mu\p X$ the jump measure of $X$, that is,
\[
\mu\p{X}(\om,\rmd t,\rmd x)=\sum_{s>0}1_{\{\Delta X_s(\om)\neq0\}}\delta_{(s,\Delta X_s(\om))}(\rmd t,\rmd x),
\]
where $\delta_a$ denotes the Dirac measure at point $a\in\Rbb\p d$. 
From \cite[Theorem II.1.16]{JS00}, $\mu\p X$ is an \emph{integer-valued random measure} with respect to $\Fbb$ (see \cite[Definition II.1.13]{JS00}). 

By $(B\p X, C\p X,\nu\p X)$ we denote the $\Fbb$-predictable characteristics of $X$ with respect to the truncation function $h(x)=1_{\{|x|\leq1\}}x$ (see \cite[Definition II.2.3]{JS00}). Recall that $\nu\p X$ is a predictable random measure characterized by the following properties: For any $\Fbb$-predictable mapping $W$ such that $|W|\ast\mu\p X\in\Ascr\p +_\mathrm{loc}$, we have $|W|\ast\nu\p X\in\Ascr\p +_\mathrm{loc}$ and $(W\ast\mu\p X-W\ast\nu\p X)\in\Hscr\p1_{\mathrm{loc},0}$ {(see \cite[Theorem II.1.8]{JS00})}.

We are now going to introduce the stochastic integral with respect to $(\mu\p X-\nu\p X)$  of an  $\Fbb$-predictable mapping  $W$.

Let $W$ be an $\Fbb$-predictable mapping. We define the process $\wt W\p X$ by
\begin{equation}\label{eq:def.wr}
\wt W_t\p X(\om):=W(\om,t,\Delta X_t(\om))1_{\{\Delta X_t(\om)\neq0\}}-\widehat W_t\p X(\om),
\end{equation}
where, for $t\geq0$,
\[
\widehat W_t\p X(\om):=\begin{cases} \displaystyle\int_{\Rbb\p d}W(\om,t,x)\nu\p X(\om,\{t\}\times\rmd x),&\quad \textnormal{if } \displaystyle\int_{\Rbb\p d}|W(\om,t,x)|\nu\p X(\om,\{t\}\times\rmd x)<+\infty;\\\\
\displaystyle+\infty,&\quad\textnormal{else}.
\end{cases}
\]
Notice that, according to \cite[Lemma II.1.25]{JS00}, $\widehat W\p X$ is predictable and a version of the predictable projection of the process $(\om,t)\mapsto W(\om,t,\Delta X_t(\om))1_{\{\Delta X_t(\om)\neq0\}}$. In symbols, denoting this latter process by $W(\cdot,\cdot,\Delta X)1_{\{\Delta X\neq0\}}$, we have  $\widehat W\p X={}^{p,\Fbb}(W(\cdot,\cdot,\Delta X)1_{\{\Delta X\neq0\}})$. So, since $\widehat W\p X$ depends on the filtration $\Fbb$ as well, we shall also write, if necessary, $\widehat W\p{X,\Fbb}$ and $\wt W\p{X,\Fbb}$ to stress the filtration. This notation will be especially used in Section \ref{sec:prop.mar.rep} below.

For $q\geq1$, we introduce (see \cite[(3.62)]{J79})
\[
\textstyle\Gscr\p q(\mu\p X):=\big\{W:\ W\textnormal{ is an } \Fbb\textnormal{-predictable function and }\ \big(\sum_{0\leq s\leq \cdot}(\wt W\p X)\p2_s\big)\p{q/2}\in \Ascr\p+\big\}.
\]
The definition of $\Gscr\p q_\mathrm{loc}(\mu\p X)$ is similar and makes use of $\Ascr\p+_\mathrm{loc}$ instead. To specify the filtration, we sometimes write $\Gscr\p q(\mu\p X,\Fbb)$. 
Setting
\[
\textstyle \|W\|_{\Gscr\p q(\mu\p X)}:=\Ebb\Big[\big(\sum_{s\geq 0}(\wt W\p X)\p2_s\big)\p{q/2}\Big]\p{1/q},
\] 
we get a semi-norm on $\Gscr\p q(\mu\p X)$. 

Let now $W\in\Gscr\p1_\mathrm{loc}(\mu\p X)$. The stochastic integral of $W$ with respect to $(\mu\p X-\nu\p X)$ is denoted by $W\ast(\mu\p X-\nu\p X)$ and is defined as the unique purely discontinuous local martingale $Z\in\Hscr\p1_\mathrm{loc,0}(\Fbb)$ such that $\Delta Z=\wt W\p X$ (up to an evanescent set). See \cite[Definition II.1.27]{JS00} and the subsequent comment. We recall that, according to \cite[Proposition 3.66]{J79}, the inclusion $W\ast(\mu\p X-\nu\p X)\in\Hscr\p q$ holds if and only if $W\in\Gscr\p q(\mu\p X)$, $q\geq1$.

For two $\Fbb$-semimartingales $X$ and $Y$, we denote by $[X,Y]$ the quadratic variation of $X$ and $Y$:
\[
[X,Y]_t:=\aPP{X\p c}{Y\p c}_t+\sum_{s\leq t}\Delta X_s\Delta Y_s,
\]
where $X\p c$ and $Y\p c$ denote the continuous local martingale part of $X$ and $Y$, respectively.
We recall that if $X,Y\in\Mloc(\Fbb)$, then $XY\in\Mloc(\Fbb)$ if and only if $[X,Y]\in\Mloc(\Fbb)$ holds. If furthermore $X,Y\in\Hscr\p2_\mathrm{loc}(\Fbb)$, then $[X,Y]\in\Ascr\p+_\mathrm{loc}(\Fbb)$ and $[X,Y]\p p=\aPP{X}{Y}$.
\paragraph{The stochastic integral for multidimensional local martingales.}Let us fix $q\geq1$ and consider an $\Rbb\p d$-valued stochastic process $X=(X\p1,\ldots,X\p d)\p{tr}$ such that $X\p i\in\Hloc\p q(\Fbb)$, $i=1,\ldots,d$. We denote by $a$ and $A$ the processes introduced in \cite[Chapter 4, Section 4\S a]{J79} such that $[X,X]=a\cdot A$. We recall that $A\in\Ascr\p+_\mathrm{loc}(\Fbb)$ and that $a$ is an optional process taking values in the space of $d$-dimensional symmetric and nonnegative matrices. Notice that, if $q=2$, then we can take $C=A\p p$ and $c$ is a predictable process taking values in the space of $d$-dimensional symmetric and nonnegative matrices such that $c\p{i,j}\cdot C=(a\p{i,j}\cdot A)\p p$. 

Let $K$ be an $\Rbb\p d$-valued measurable process and define $\|K\|_{\Lrm\p q(X)}=\Ebb[((K\p{tr}aK)\cdot A_\infty)\p{q/2}]$. We denote by $\Lrm\p q(X)$ the space of $\Rbb\p d$-valued and $\Fbb$-predictable processes $K$ such that $\|K\|_{\Lrm\p q(X)}<+\infty$. Notice that $K\in\Lrm\p q(X)$ if and only if $K$ is $\Rbb\p d$-valued, $\Fbb$-predictable and $((K\p{tr}aK)\cdot A)\p{q/2}\in\Ascr\p+$. The space $\Lrm\p{q}_\mathrm{loc}(X)$ is defined analogously but making use of $\Ascr\p+_\mathrm{loc}$ instead of $\Ascr\p+$.
For $K\in\Lrm_\mathrm{loc}\p1(X)$, we denote by $K\cdot X$ the stochastic integral of $K$ with respect to $X$. Recall that $K\cdot X\in\Hscr\p1_{\mathrm{loc},0}$ and that it is \emph{always} an $\Rbb$-valued process. Sometimes, to stress the underlying filtration, we write  $\Lrm\p q(X,\Fbb)$ or $\Lrm\p q_\mathrm{loc}(X,\Fbb)$. 

We observe that if $X\in\Mloc(\Fbb)$ (in particular, $X$ is $\Rbb$-valued), then we can chose $a=1$, $A=[X,X]$ and we get the usual definition of the stochastic integral with respect to $X$ (see \cite[Definition 2.46]{J79}). If furthermore $X$ is of finite variation and $K\in \Lrm_\mathrm{loc}\p1(X)$, then the stochastic integral $K\cdot X$ coincides with the Stieltjes-Lebesgue integral, whenever this latter one exists and is finite.

\section{Martingale Representation in the Independently Enlarged Filtration}\label{sec:prop.mar.rep}
We start this section introducing the notion of the \emph{weak representation property} (abbreviated by WRP).
\begin{definition}\label{def:w.prp} Let $\Fbb$ be a right-continuous filtration and $(X,\Fbb)$  an $\Rbb\p d$-valued semimartingale with continuous local martingale part $X\p c$ and predictable $\Fbb$-characteristics $(B\p X,C\p X,\nu\p X)$. We say that $X$ possesses the WRP with respect to $\Fbb$ if every $N\in\Hloc\p 1(\Fbb)$ can be represented as
\begin{equation}\label{eq:wprp}
N_t=N_0+K\cdot X\p c_t+W\ast(\mu\p X-\nu\p X)_t,\quad t\geq0,\quad K\in\Lrm_\mathrm{loc}\p 1(X\p c,\Fbb),\quad W\in\Gloc\p 1(\mu\p X,\Fbb).
\end{equation}
\end{definition}
For our aims, the following characterization of the WRP will be useful:
\begin{proposition}\label{prop:eq.wprp}
The $\Rbb\p d$-valued semimartingale $X$ possesses the \textnormal{WRP} with respect to $\Fbb$ if and only if every $N\in\Hscr\p2(\Fbb)$ has the representation
\begin{equation}\label{eq:sq.int.mar}
N_t=N_0+K\cdot X\p c_t+W\ast(\mu\p X-\nu\p X)_t,\quad t\geq0,\quad K\in\Lrm\p 2(X\p c,\Fbb),\quad W\in\Gscr\p 2(\mu\p X,\Fbb).
\end{equation}
\end{proposition} 
\begin{proof}
By localization it is enough to show that \eqref{eq:sq.int.mar} holds if and only if every $N\in\Hscr\p1(\Fbb)$ can be represented as
\[
N_t=N_0+K\cdot X\p c_t+W\ast(\mu\p X-\nu\p X)_t,\quad t\geq0,\quad K\in\Lrm\p 1(X\p c,\Fbb),\quad W\in\Gscr\p 1(\mu\p X,\Fbb).
\]
But this is just \cite[Proposition 3.2]{DT19}. The proof is complete.
\end{proof}

\subsection{Propagation of the Weak Representation Property}\label{subsec:Weak-Strong}

Let $\Fbb=(\Fscr_t)_{t\geq0}$ and $\Hbb=(\Hscr_t)_{t\geq0}$ be right-continuous filtrations.  In this section we consider the filtration $\Gbb:=\Fbb\vee\Hbb$, that is, $\Gbb$ is the smallest filtration containing both $\Fbb$ and $\Hbb$. We then make the following assumption:
\begin{assumption}\label{ass:indep}
The filtrations $\Fbb$ and $\Hbb$ are independent.
\end{assumption}
We recall that two filtrations $\Fbb=(\Fscr_t)_{t\geq0}$ and $\Hbb=(\Hscr_t)_{t\geq0}$ are called independent, if the $\sig$-algebras $\Fscr_\infty$ and $\Hscr_\infty$ are independent.

As a first consequence of Assumption \ref{ass:indep}, we get that the filtration $\Gbb$ is again right-continuous (see Lemma \ref{lem:prop} (i)). Furthermore, the following result holds:
\begin{proposition}\label{prop:G.char}
Let $\Fbb$ and $\Hbb$ satisfy  Assumption \ref{ass:indep}. Let $(X,\Fbb)$ be an $\Rbb\p d$-valued semimartingale and let $(B\p X,C\p X,\nu\p X)$ denote the $\Fbb$-predictable characteristics of $X$. Then, $(X,\Gbb)$ is a semimartingale and the $\Gbb$-predictable characteristics of $X$ are again given by $(B\p X,C\p X,\nu\p X)$.
\end{proposition}
\begin{proof}
The result follows from Lemma \ref{lem:prop} (ii) and \cite[Theorem II.2.21]{JS00}. The proof is complete.
\end{proof}
In the proof of Theorem \ref{thm:mar.rep.G} below, we need the following technical proposition: 
\begin{proposition}\label{prop:gilocG}
Let $\Fbb$ and $\Hbb$ satisfy  Assumption \ref{ass:indep}. Let $(X,\Fbb)$ be an $\Rbb\p d$-valued semimartingale with jump-measure $\mu\p X$ and $\Fbb$-predictable characteristics $(B\p X,C\p X,\nu\p X)$. Then, for every $W\in\Gloc\p1(\mu\p X,\Fbb)$ we have:

\textnormal{(i)} $\wh W\p{X,\Fbb}=\wh W\p{X,\Gbb}$.

\textnormal{(ii)} $\wt W\p{X,\Fbb}=\wt W\p{X,\Gbb}$.

\textnormal{(iii)} $W\in\Gloc\p1(\mu\p X,\Gbb)$. If moreover $W\in\Gscr\p q(\mu\p X,\Fbb)$, $q\geq1$, then  $W\in\Gscr\p q(\mu\p X,\Gbb)$.
\end{proposition}
\begin{proof}
First we notice that, by Proposition \ref{prop:G.char}, $X$ is a $\Gbb$-semimartingale. Hence, (ii) is a direct consequence of (i) and of the definition of $\wt W\p{X,\Gbb}$. Furthermore, (iii) follows immediately from (ii). We now show (i). Because of Proposition \ref{prop:G.char}, the $\Gbb$-predictable compensator of $\mu\p X$ is again given by $\nu\p{X}$. Therefore, by the definition of $\wh W\p{X,\Gbb}$ we have the identity $\wh W\p{X,\Gbb}=\wh W\p{X,\Fbb}$. The proof is complete.
\end{proof}

Let $(E,\Escr)$ be a measurable space endowed by the $\sig$-algebra $\Escr$. We denote by $\Bb(\Escr)$ the space of $\Rbb$-valued $\Escr$--$\Bscr(\Rbb)$-measurable and bounded functions on $E$. We now come to the main result of the present paper. 
\begin{theorem}\label{thm:mar.rep.G}
Let $(X,\Fbb)$ be an $\Rbb\p d$-valued semimartingale with jump-measure $\mu\p X$ and $\Fbb$-predictable characteristics $(B\p X,C\p X,\nu\p X)$. Let $(Y,\Hbb)$ be an $\Rbb\p \ell$-valued semimartingale with jump-measure $\mu\p Y$ and $\Hbb$-predictable characteristics $(B\p Y,C\p Y,\nu\p Y)$. Let furthermore $X$ have the \emph{WRP} with respect to $\Fbb$ and let $Y$ have the \emph{WRP} with respect to $\Hbb$. If $\Fbb$ and $\Hbb$ satisfy Assumption \ref{ass:indep}, then the $\Rbb\p{d}\times\Rbb\p\ell$-valued $\Gbb$-semimartingale $Z=(X,Y)\p{tr}$ possesses the \emph{WRP} with respect to $\Gbb$, where $\Gbb:=\Fbb\vee\Hbb$.
\end{theorem}
\begin{proof}
Let $\xi$ belong to $\Bb(\Fscr_\infty)$ and let $\et$ belong to $\Bb(\Hscr_\infty)$. We consider the martingales $M\in\Hscr\p2(\Fbb)$ and $N\in\Hscr\p2(\Hbb)$ such that $M_t=\Ebb[\xi|\Fscr_t]$ and $N_t=\Ebb[\et|\Hscr_t]$ a.s.\ for every $t\geq0$. Since $X$ has the WRP with respect to $\Fbb$ and $Y$ has the WRP with respect to $\Hbb$, because of Proposition \ref{prop:eq.wprp}, we can represent $M$ as 
\begin{equation}\label{eq:rep.Y}
M_t=M_0+K\cdot X_t\p c+W\ast(\mu\p X-\nu\p X)_t,\quad t\geq0,\quad K\in\Lrm\p2(X\p c,\Fbb),\quad W\in\Gscr\p2(\mu\p X,\Fbb)
\end{equation} 
and $N$ as
\begin{equation}\label{eq:rep.Z}
N_t=N_0+J\cdot Y_t\p c+V\ast(\mu\p Y-\nu\p Y)_t,\quad t\geq0,\quad J\in\Lrm\p2(Y\p c,\Hbb),\quad V\in\Gscr\p2(\mu\p Y,\Hbb).
\end{equation}
Furthermore, because Lemma \ref{lem:prop} (ii) and (iii), the processes $M$, $N$ and $MN$ are $\Gbb$-martingales and they are bounded. Hence, we have, in particular, $M,N,MN\in\Hscr\p2(\Gbb)$. By Proposition \ref{prop:G.char},  $(X,\Gbb)$ and $(Y,\Gbb)$ are semimartingales and their $\Gbb$-predictable characteristics are given by $(B\p X,C\p X,\nu\p X)$ and $(B\p Y,C\p Y,\nu\p Y)$, respectively. 
In conclusion, \eqref{eq:rep.Y} and \eqref{eq:rep.Z} are also valid in $\Gbb$. We can then apply the formula of integration by parts with respect to the enlarged filtration $\Gbb$ and use \eqref{eq:rep.Y}, \eqref{eq:rep.Z}, \cite[Proposition II.1.30 b) and III.4.9 in Theorem III.4.5]{JS00} to obtain
\begin{equation}\label{eq:par.int}
\begin{split}
M_tN_t&=M_0N_0+N_-\cdot M_t+M_-\cdot N_t+[M,N]_t
\\&=M_0N_0+N_-K\cdot X\p c_t+N_-W\ast(\mu\p X-\nu\p X)_t+M_-J\cdot Y\p c_t+M_-V\ast(\mu\p Y-\nu\p Y)_t+[M,N]_t.
\end{split}
\end{equation}
We split the remaining part of the proof in several steps.
\paragraph*{Step 1: Representation of $[M,N]$.} 
By Lemma \ref{lem:prop} (ii), the process $M\p cN\p c$ is a continuous local martingale and hence $\aPP{M\p c}{N\p c}=0$. By definition of the quadratic co-variation, from Proposition \ref{prop:gilocG} (ii) and the identities \eqref{eq:rep.Y} and \eqref{eq:rep.Z}, we therefore get
\begin{equation}\label{eq:rep.sq.br}
[M,N]_t=\sum_{0\leq s\leq t}\wt W\p{X,\Fbb}_s\wt V\p{Y,\Hbb}_s=\sum_{0\leq s\leq t}\wt W\p{X,\Gbb}_s\wt V\p{Y,\Gbb}_s.
\end{equation}
Since $MN\in\Hscr\p2(\Gbb)$, we have that $S:=[M,N]\in\Mloc(\Gbb)$. Moreover,  Kunita--Watanabe's inequality (see \cite[Theorem 8.3, Eq.\ (3.1)]{HWY92}) and  the independence of $\Fbb$ and $\Hbb$ yield 
\[\begin{split}
\Ebb\bigg[\sup_{t\geq0}\big|[M,N]_t\big|\p2\bigg]&\leq\Ebb\big[\Var\big([M,N]\big)_\infty\p2\big]\leq\Ebb\big[[M,M]_\infty[N,N]_\infty\big]=\Ebb\big[[M,M]_\infty\big]\Ebb\big[[N,N]_\infty\big]<+\infty,
\end{split}\]
meaning that the $\Gbb$-local martingale $S$ belongs to $\Hscr\p2(\Gbb)$ as well. By the definition of $S$ and Proposition \ref{prop:gilocG} (ii), we have 
\begin{equation}\label{eq:ju.Z}
\Delta S=\wt W\p{X,\Fbb}\wt V\p{Y,\Hbb}=\wt W\p{X,\Gbb}\wt V\p{Y,\Gbb}.
\end{equation}
By the definition of $\wt W\p{X,\Gbb}$ and $\wt V\p{Y,\Gbb}$ and Proposition \ref{prop:gilocG} (i), we have
\begin{equation}\label{eq:exp.for.prod}
\begin{split}
\wt W\p{X,\Gbb}\wt V\p{Y,\Gbb}&=\wt W\p{X,\Fbb}\wt V\p{Y,\Hbb}=\Big(W(\cdot,\cdot,\Delta X)1_{\{\Delta X\neq0\}}-\wh W\p{X,\Gbb}\Big)\Big(V(\cdot,\cdot,\Delta Y)1_{\{\Delta Y\neq0\}}-\wh V\p{Y,\Gbb}\Big)
\\&=W(\cdot,\cdot,\Delta X)V(\cdot,\cdot,\Delta Y)1_{\{\Delta X\neq0,\Delta Y\neq0\}}-\wh W\p{X,\Gbb}V(\cdot,\cdot,\Delta Y)1_{\{\Delta Y\neq0\}}\\&\qquad-\wh V\p{Y,\Gbb}W(\cdot,\cdot,\Delta X)1_{\{\Delta X\neq0\}}+\wh V\p{Y,\Gbb}\wh W\p{X,\Gbb}.
\end{split}
\end{equation}
We are now going to compute the $\Gbb$-predictable projection of $\wt W\p{X,\Gbb}\wt V\p{Y,\Gbb}$. We recall that the processes $\wh W\p{X,\Gbb}$ and $\wh V\p{Y,\Gbb}$ are $\Gbb$-predictable, by \cite[Lemma II.1.25]{JS00}. Furthermore, since $\wt W\p{X,\Gbb}\in\Gscr\p2(\mu\p X,\Gbb)$ and  $\wt V\p{Y,\Gbb}\in\Gscr\p2(\mu\p Y,\Gbb)$ by Proposition \ref{prop:gilocG} (iii), we also have that $\wh W\p{X,\Gbb}$ and $\wh V\p{Y,\Gbb}$ are finite valued. By the $\Gbb$-martingale property of $S$, we have ${}\p{p,\Gbb}{(\Delta S)}=0$ (see \cite[Corollary I.2.31]{JS00}). Therefore, using \eqref{eq:ju.Z}, \cite[Theorem I.2.28 c), Corollary I.2.31 and Lemma II.1.25]{JS00} in \eqref{eq:exp.for.prod}, since all involved processes are finite-valued, we can compute
\[
\begin{split}
0={}\p{p,\Gbb}(\wt W\p{X,\Gbb}\wt V\p{Y,\Gbb})&={}\p{p,\Gbb}\big(W(\cdot,\cdot,\Delta X)V(\cdot,\cdot,\Delta Y)1_{\{\Delta X\neq0,\Delta Y\neq0\}}\big)-{}\p{p,\Gbb}\big(V(\cdot,\cdot,\Delta Y)1_{\{\Delta Y\neq0\}}\big)\wh W\p{X,\Gbb}\\&\qquad-{}\p{p,\Gbb}\big(W(\cdot,\cdot,\Delta X)1_{\{\Delta X\neq0\}}\big)\wh V\p{Y,\Gbb}+\wh W\p{X,\Gbb}\wh V\p{Y,\Gbb}
\\&={}\p{p,\Gbb}\big(W(\cdot,\cdot,\Delta X)V(\cdot,\cdot,\Delta Y)1_{\{\Delta X\neq0,\Delta Y\neq0\}}1_{\{\Delta Z\neq0\}}\big)-\wh W\p{X,\Gbb}\wh V\p{Y,\Gbb},
\end{split}
\]
which yields
\begin{equation}\label{eq:proj}
\wh W\p{X,\Gbb}\wh V\p{Y,\Gbb}\quad\textnormal{\ is a version of\ }\quad {}\p{p,\Gbb}\big(W(\cdot,\cdot,\Delta X)V(\cdot,\cdot,\Delta Y)1_{\{\Delta X\neq0,\Delta Y\neq0\}}1_{\{\Delta Z\neq0\}}\big).
\end{equation}
We set $F(\om,t,x,y):= W(\om,t,x)V(\om,t,y)1_{\{x\neq0,y\neq0\}}$. Then, $F$ is $\Pscr(\Gbb)\otimes\Bscr(\Rbb\p d)\otimes(\Rbb\p\ell)$-measurable. By \cite[Lemma II.1.25]{JS00} and \eqref{eq:proj}, we get that $\wh W\p{X,\Gbb}\wh V\p{Y,\Gbb}$ is a version of $\wh {F}\p{\,Z,\Gbb}$. We now define
\begin{equation}\label{eq:def.int.U}
U(\om,t,x,y):=F(\om,t,x,y)-\wh W\p{X,\Gbb}_t(\om)V(\om,t,y)1_{\{y\neq0\}}-\wh V\p{Y,\Gbb}_t(\om)W(\om,t,x)1_{\{x\neq0\}}
\end{equation}
and clearly have
\[
\begin{split}
U(\om,t,\Delta X_t(\om),\Delta Y_t(\om))1_{\{\Delta Z_t(\om)\neq0\}}&=
W(\om,t,\Delta X_t(\om))V(\om,t,\Delta Y_t(\om))1_{\{\Delta X_t(\om)\neq0,\Delta Y_t(\om)\neq0\}}\\&-\wh W\p{X,\Gbb}_t(\om)V(\om,t,\Delta Y_t(\om))1_{\{\Delta Y_t(\om)\neq0\}}\\&-\wh V\p{Y,\Gbb}_t(\om)W(\om,t,\Delta X_t(\om))1_{\{\Delta X_t(\om)\neq0\}}.
\end{split}
\]
From this, by \cite[Theorem I.2.28 c)]{JS00} and \eqref{eq:proj}, we see that $-\wh W\p{X,\Gbb}\wh V\p{Y,\Gbb}$ is a version of the $\Gbb$-predictable projection of the process $ U(\cdot,\cdot,\Delta X,\Delta Y)1_{\{\Delta Z\neq0\}}$. Hence, by \cite[Lemma II.1.25]{JS00}, we obtain
\begin{equation}\label{eq:Util}
\begin{split}
\wt U\p{Z,\Gbb}_t(\om)&:=U(\om,t,\Delta X_t(\om),\Delta Y_t(\om))1_{\{\Delta Z_t(\om)\neq0\}}-\wh U\p{Z,\Gbb}_t(\om)
\\&=U(\om,t,\Delta X_t(\om),\Delta Y_t(\om))1_{\{\Delta Z_t(\om)\neq0\}}+\wh W\p{X,\Gbb}_t(\om)\wh V\p{Y,\Gbb}_t(\om)
\\&=\wt W\p{X,\Gbb}_t(\om)\wt V\p{Y,\Gbb}_t(\om),\quad \textnormal{for every $t\in\Rbb_+$,\ a.s.\ }
\end{split}
\end{equation}
where, in the last identity, we used \eqref{eq:exp.for.prod}. We therefore have
\[
\Ebb\left[\sum_{s\geq0}\big(\wt U\p{Z,\Gbb}_s\big)\p2\right]=\Ebb\left[\sum_{s\geq0}\big(\wt W\p{X,\Gbb}_s\wt V\p{Y,\Gbb}_s\big)\p2\right]=\Ebb\big[[S,S]_\infty\big]<+\infty,
\]
where, for the last estimate, we recall that $S\in\Hscr\p2(\Gbb)$ (see \cite[Proposition I.4.50 c)]{JS00}). Therefore, the inclusion $U\in\Gscr\p2(\mu\p{Z},\Gbb)$ holds. Hence, we can introduce the purely-discontinuous square-integrable $\Gbb$-martingale $U\ast(\mu\p{Z}-\nu\p{Z})$ and we get
\[
\Delta U\ast(\mu\p{Z}-\nu\p{Z})=\wt U\p{Z,\Gbb}=\wt W\p{X,\Gbb}\wt V\p{Y,\Gbb}=\Delta S,
\] 
up to an evanescent set. Hence, the purely discontinuous martingales $U\ast(\mu\p{Z}-\nu\p{Z})$ and $S=[M,N]$ have the same jumps, up to an evanescent set. By \cite[Corollary I.4.19]{JS00}, we conclude that $S$ and $U\ast(\mu\p{Z}-\nu\p{Z})$ are indistinguishable. Summarizing, we have shown that $U\in\Gscr\p2(\mu\p{Z},\Gbb)$ and
\begin{equation}\label{eq:rep.sq.br.rep}
[M,N]=U\ast(\mu\p{Z}-\nu\p{Z}).
\end{equation}
The proof of Step 1 is complete.
\paragraph*{Step 2: Representation of $N_-W\ast(\mu\p X-\nu\p X)$ and of $M_-V\ast(\mu\p Y-\nu\p Y)$.}
For notational convenience, we denote in this part $g_1(x,y):=1_{\{x\neq0\}}$ and $g_2(x,y):=1_{\{y\neq0\}}$, $(x,y)\in\Rbb\p d\times\Rbb\p\ell$. We are only going to show how to represent $N_-W\ast(\mu\p X-\nu\p X)$, the result for $M_-V\ast(\mu\p Y-\nu\p Y)$ being completely analogous. We observe that, by Lemma \ref{lem:What}, we obtain the identity  $W\ast(\mu\p X-\nu\p X)=Wg_1\ast(\mu\p{Z}-\nu\p{Z})$. By Proposition \ref{prop:G.char}, Lemma \ref{lem:What}, $N_-$ being predictable, we have, because of Proposition \ref{prop:gilocG} (ii), (iii) and the independence,
\[
\begin{split}
\Ebb\bigg[\sum_{s\geq0}\big(\wt{N_{-}Wg_1}_s\p{Z,\Gbb}\big)\p2\bigg]&=\Ebb\bigg[\sum_{s\geq0}\big(N_{s-}\wt{Wg_1}_s\p{Z,\Gbb}\big)\p2\bigg]
\\&=\Ebb\bigg[\sum_{s\geq0}\big(N_{s-}\wt{W}_s\p{X,\Gbb}\big)\p2\bigg]
\\&\leq\Ebb\bigg[\sup_{t\geq0}(N\p2_t)\sum_{s\geq0}\big(\wt{W}_s\p{X,\Fbb}\big)\p2\bigg]
\\&\leq\|N\|_{\Hscr\p2(\Gbb)}\p2\|W\|_{\Gscr\p2(\mu\p X,\Gbb)}\p2<+\infty,
\end{split}
\]
 showing the inclusion $N_-Wg_1\in\Gscr\p2(\mu\p{Z},\Gbb)$. \cite[Proposition II.1.30b)]{JS00} and Lemma \ref{lem:What} (i) furthermore yield
\begin{equation}\label{eq:rep.first.st.int}
\begin{split}
N_-W\ast(\mu\p X-\nu\p X)&=N_-\cdot(W\ast(\mu\p X-\nu\p X))\\&=N_-\cdot(Wg_1\ast(\mu\p{Z}-\nu\p{Z}))\\&=N_-Wg_1\ast(\mu\p{Z}-\nu\p{Z}).
\end{split}
\end{equation}
Analogously, we get $M_-Vg_2\in\Gscr\p2(\mu\p{Z},\Gbb)$ and
\begin{equation}\label{eq:rep.sec.st.int}
M_-V\ast(\mu\p Y-\nu\p Y)=M_-Vg_2\ast(\mu\p{Z}-\nu\p{Z}).
\end{equation}
The proof of Step 2 is complete.
\paragraph*{Step 3: Representation of the continuous part.} We stress that for this step we need the properties of the stochastic integral of multidimensional predictable processes with respect to a \emph{multidimensional continuous local martingale}. For this topic we refer to \cite[III.\S4a]{JS00} and \cite[Section 4.2 \S a\ and \S b]{J79}.

We consider the $\Rbb\p{d}\times\Rbb\p{\ell}$-valued continuous $\Gbb$-local martingale $Z\p c=(X\p c, Y\p c)\p{tr}$ and introduce the $\Gbb$-predictable $\Rbb\p{d}\times\Rbb\p{\ell}$-valued process $H=(N_-K,M_-J)\p{tr}$. We notice that the identities $(Z\p i)\p c=(X\p i)\p c$, $H\p i=N_-K\p i$ for $i=1,\ldots,d$ and $(Z\p i)\p c=(Y\p i)\p c$, $H\p i=N_-K\p i$ for $i=d+1,\ldots,d+\ell$ holds. Let us define $C:=A+B$, where $A:=\sum_{i=1}\p d\aPP{(X\p i)\p c}{(X\p i)\p c}$ and $B:=\sum_{i=1}\p \ell\aPP{(Y\p i)\p c}{(Y\p i)\p c}$. So, $A$ and $B$ are absolutely continuous with respect to $C$. Additionally, because of Lemma \ref{lem:prop} (ii) and the uniqueness of the point brackets, $\aPP{(X\p i)\p c}{(X\p i)\p c}$ does not change in $\Gbb$. The same holds for $\aPP{(Y\p i)\p c}{(Y\p i)\p c}$. Therefore, $A$ is $\Fbb$-predictable and $B$ is $\Hbb$-predictable. According to \cite[Section 4.2 \S a\ and \S b]{J79}, there exists a $\Gbb$-predictable process $c$ taking values in the set of nonnegative symmetric  ${(d+\ell)\times(d+\ell)}$-matrices such that $\aPP{(Z\p i)\p c}{(Z\p j)\p c}=c\p{i,j}\cdot C$. Analogously, we find an $\Fbb$-predictable processes $a$ taking values in the set of nonnegative symmetric ${d\times d}$-matrices such that $\aPP{(X\p i)\p c}{(X\p j)\p c}=a\p{i,j}\cdot A$, $i,j=1,\ldots, d$ and an $\Hbb$-predictable processes $b$ taking values in the set of nonnegative symmetric ${\ell\times \ell}$-matrices such that $\aPP{(Y\p i)\p c}{(Y\p j)\p c}=b\p{i,j}\cdot B$, $i,j=1,\ldots, \ell$. Furthermore, since $\aPP{(X\p i)\p c}{(Y\p j)\p c}=0$, $i=1,\ldots,d$, $j=1,\ldots,\ell$, we also have $c\p{i,d+j}=0$, $i=1,\ldots,d$, $j=1,\ldots,\ell$.  
We are going to show that the $\Gbb$-predictable process $H$ introduced above belongs to $\Lrm\p2(Z\p c,\Gbb)$. For this, according to \cite[III.4.3]{JS00}, we have to verify that the increasing process $H\p\top c H\cdot C$ is integrable. Because of the structure of $c$, we see that the identity
\[
\begin{split}
H\p{tr} c H&=\sum_{i,j=1}\p{d+\ell}H\p ic\p{i,j}H\p j=\sum_{i,j=1}\p{d}H\p ic\p{i,j}H\p j+\sum_{i,j=1}\p{\ell}H\p {d+i}c\p{d+i,d+j}H\p{d+j}\\&=\sum_{i,j=1}\p{d}N_-\p2K\p iK\p jc\p{i,j}+\sum_{i,j=1}\p{\ell}M_-\p2J\p {i}J\p jc\p{d+i,d+j}
\end{split}
\]
holds. By linearity of the integral with respect to $C$, we now get
\[
\begin{split}
(H\p{tr} c H)\cdot C&=\sum_{i,j=1}\p{d}\{N_-\p2K\p iK\p jc\p{i,j}\cdot C\}+\sum_{i,j=1}\p{\ell}\{M_-\p2J\p {i}J\p jc\p{d+i,d+j}\cdot C\}
\\&=\sum_{i,j=1}\p{d}\{N_-\p2K\p iK\p j\cdot\aPP{(X\p i)\p c}{(X\p j)\p c}\}+\sum_{i,j=1}\p{\ell}\{M_-\p2J\p {i}J\p j\cdot\aPP{(Y\p i)\p c}{(Y\p j)\p c}\}
\\&=\sum_{i,j=1}\p{d}\{N_-\p2K\p iK\p ja\p{i,j}\cdot A\}+\sum_{i,j=1}\p{\ell}\{M_-\p2J\p {i}J\p jb\p{i,j}\cdot B\}
\\&=(N_-\p2K\p{tr} aK)\cdot A+(M_-\p2J\p{tr} bJ)\cdot B.
\end{split}
\]
Hence, the independence of $\Fbb$ and $\Hbb$ yields
\[\begin{split}\Ebb\big[(H\p{tr} c H)\cdot C_\infty\big]&=
\Ebb\big[(N_-\p2K\p{tr} aK)\cdot A_\infty+(M_-\p2J\p{tr} bJ)\cdot B_\infty\big]\\&\leq\Ebb\Big[\sup_{t\geq0}N\p2_t\Big]\|K\|_{\Lrm\p2(X\p c,\Fbb)}\p2+\Ebb\Big[\sup_{t\geq0}M\p2_t\Big]\|J\|_{\Lrm\p2(Y\p c,\Hbb)}\p2<+\infty,
\end{split}\]
where, in the last estimate, we used that $M$ and $N$ are square integrable martingales, that $K\in \Lrm\p2(X\p c,\Fbb)$ and that $J\in \Lrm\p2(Y\p c,\Hbb)$. This shows the inclusion $H\in \Lrm\p2(Z\p c,\Gbb)$. We are now going to verify that the identity $H\cdot Z\p c=N_-K\cdot X\p c+M_-J\cdot Y\p c$ holds. Let $R\in\Hloc\p2(\Gbb)$ be continuous. Then, there exist two $\Gbb$-predictable process $a\p{Ri}$ and $b\p{R\p j}$ such that $\aPP{R}{(X\p i)\p c}=a\p{Ri}\cdot A$ and $\aPP{R}{(Y\p j)\p c}=b\p{Rj}\cdot B$, $i=1,\ldots,d$, $j=1,\ldots,\ell$ (see \cite[Eq.\ (III.4.4) and the explanation before]{JS00}). On the other side, since $A$ and $B$ are absolutely continuous with respect to $C$, we find two $\Gbb$-predictable processes $\al$ and $\bt$ such that $A=\al\cdot C$ and $B=\bt\cdot C$. In conclusion, we get $\aPP{R}{(X\p i)\p c}=a\p{Ri}\al\cdot C$ and $\aPP{R}{(Y\p j)\p c}=b\p{Rj}\bt\cdot C$, $i=1,\ldots,d$, $j=1,\ldots,\ell$. So, if we define the $\Gbb$-predictable process $c\p{R k}$, $i=1,\ldots,d+\ell$, by
\[
c\p{Ri}:=\begin{cases}a\p{Ri}\al,\quad i=1,\ldots,d\\\\b\p{Ri-d}\bt, \quad i=d+1,\ldots,\ell,\end{cases}
\]
we get $\aPP{R}{(Z\p i)\p c}=c\p{Ri}\cdot C$, $i=1,\ldots,d+l$. By the linearity of the predictable quadratic covariation and \cite[Theorem III.4.5b)]{JS00} applied to the two stochastic integrals $N_-K\cdot X\p c$ and $M_-J\cdot Y\p c$, we compute
\[
\begin{split}
\aPP{R}{N_-K\cdot X\p c+M_-J\cdot Y\p c}&=\aPP{R}{N_-K\cdot X\p c}+\aPP{R}{M_-J\cdot Y\p c}\\&=\sum_{i=1}\p dN_-K\p ia\p{Ri}\cdot A+\sum_{i=1}\p \ell M_-J\p ib\p{Ri}\cdot B
\\&=\bigg(\sum_{i=1}\p dN_-K\p ia\p{Ri}\al+\sum_{i=1}\p \ell M_-J\p ib\p{Ri}\bt\bigg)\cdot C
\\&=\bigg(\sum_{i=1}\p dN_-K\p ic\p{Ri}+\sum_{i=1}\p \ell M_-J\p ic\p{Rd+i}\bigg)\cdot C
\\&=\bigg(\sum_{i=1}\p{d+\ell}H\p ic\p{Ri}\bigg)\cdot C=\aPP{R}{H\cdot Z\p c}
\end{split}
\]
where in the last identity we again applied \cite[Theorem III.4.5b)]{JS00}. This latter computation together with \cite[Theorem III.4.5b)]{JS00} yields that $H\cdot Z\p c$ and $N_-K\cdot X\p c+M_-J\cdot Y\p c$ are indistinguishable, because $R$ was chosen arbitrarily. The proof of Step 3 is complete.
\paragraph*{Step 4: Representation of $\xi\et$.}
We now introduce the $\Pscr(\Gbb)\otimes\Bscr(\Rbb\p d\times\Rbb\p\ell)$-measurable function $(\om,t,x,y)\mapsto G(\om,t,x,y)$ by setting, for $(\om,t,x,y)\in\Om\times\Rbb_+\times\Rbb\p d\times\Rbb\p\ell$,
\[
G(\om,t,x,y):=N_{t-}(\om)W(\om,t,x)g_1(x,y)+M_{t-}(\om)V(\om,t,y)g_2(x,y)+U(\om,t,x,y), 
\]
where the $\Gbb$-predictable function $U$ has been defined in \eqref{eq:def.int.U} and the functions $g_1$ and $g_2$ have been defined at the beginning of Step 2. Step 1 and Step 2 yield the inclusion $G\in\Gscr\p2(\mu\p{Z},\Gbb)$ and, by the linearity of the stochastic integral with respect to $\mu\p Z-\nu\p Z$, we get the identity
\begin{equation}\label{eq:rep.jumps}
N_-W\ast(\mu\p X-\nu\p X)+M_-V\ast(\mu\p Y-\nu\p Y)+[M,N]=G\ast(\mu\p{Z}-\nu\p{Z}).
\end{equation}
Let us now consider the $\Rbb\p{d}\times\Rbb\p\ell$-valued continuous local martingale $Z\p c=(X\p c,Y\p c)\p{tr}$. By Step 3 we have the identity $H\cdot Z\p c=N_-K\cdot X\p c+M_-J\cdot Y\p c$, where $H\in \Lrm\p2(Z\p c,\Gbb)$  has been defined in Step 3. Therefore, form \eqref{eq:rep.jumps} and \eqref{eq:par.int}, we get $M_tN_t=M_0N_0+H\cdot Z\p c_t+G\ast(\mu\p{Z}-\nu\p{Z})_t$, $t\geq0$. Now, using that each term on the right- and on the left-hand side in this latter expression belongs to $\Hscr\p2(\Gbb)$, taking the limit $t\rightarrow+\infty$, by the martingale convergence theorem (see \cite[Theorem I.4.42a)]{JS00}), we get
\[
\xi\et=\Ebb[\xi\et|\Gscr_0]+H\cdot Z\p c_\infty+G\ast(\mu\p Z-\nu\p Z)_\infty,\quad H\in \Lrm\p2(Z\p c,\Gbb),\quad G\in\Gscr\p2(\mu\p{Z},\Gbb),
\]
 where we used the independence to write $M_0N_0=\Ebb[\xi\et|\Gscr_0]$. The proof of Step 4 is complete.
\paragraph*{Step 5: Representation of bounded $\Gscr_\infty$-measurable random variables.} As an application of the monotone class theorem, we now show that every $\xi\in\Bb(\Gscr_\infty)$ can be represented as
\begin{equation}\label{eq:rep.bou.rv}
\xi=\Ebb[\xi|\Gscr_0]+K\cdot Z\p c_\infty+W\ast(\mu\p{Z}-\nu\p{Z})_\infty,
\end{equation}
where $K\in\Lrm\p2(Z\p c,\Gbb)$ and $W\in\Gscr\p2(\mu\p{Z},\Gbb)$. To this aim, we denote by $\Kscr$ the linear space of random variables $\xi\in\Bb(\Gscr_\infty)$ that can be represented as in \eqref{eq:rep.bou.rv}. We denote  by $\Cscr\subseteq\Bb(\Gscr_\infty)$  the family  $\Cscr:=\{\xi\et,\ \xi\in\Bb(\Fscr_\infty),\ \et\in\Bb(\Hscr_\infty)\}$. Then, $\Cscr$ is clearly stable under multiplication and $\sig(\Cscr)=\Gscr_\infty$. Furthermore, by Step 4, we have $\Cscr\subseteq\Kscr$. The linear space $\Kscr$ is a monotone class of $\Bb(\Gscr_\infty)$. Indeed, let $(\xi_n)_n\subseteq\Kscr$ be a uniformly bounded sequence such that $\xi_n\geq0$ and $\xi_n\uparrow\xi$ pointwise. Then, $\xi$ is bounded and, by dominated convergence, we get $\xi_n\longrightarrow\xi$ in $L\p2(\Om,\Gscr_\infty,\Pbb)$ as $n\rightarrow+\infty$. From Lemma \ref{lem:WRP.lim}, we immediately get that $\xi\in\Kscr$. Since the inclusion $1\in\Kscr$ obviously holds, we see that $\Kscr$ is a monotone class of $\Bb(\Gscr_\infty)$. The monotone class theorem for functions (see \cite[Theorem 1.4]{HWY92}), now yields the inclusion $\Bb(\Gscr_\infty)\subseteq\Kscr$. The proof of Step 5 is complete.
\paragraph*{Step 6: Approximation and conclusion.} Let now $\xi\in L\p2(\Om,\Gscr_\infty,\Pbb)$ be a nonnegative random variable. Then, the random variable $\xi_n:=\xi\wedge n$, $n\geq1$, is bounded and furthermore $\xi_n\longrightarrow\xi$ in $L\p2(\Om,\Gscr_\infty,\Pbb)$ as $n\rightarrow+\infty$. By Step 5, $\xi_n$ can be represented as in \eqref{eq:rep.bou.rv}, for every $n\geq1$. By Lemma \ref{lem:WRP.lim}, we get that the same holds for $\xi$. If now $\xi\in L\p2(\Om,\Gscr_\infty,\Pbb)$ is an arbitrary 
random variable, we write $\xi$ as the difference of the positive and negative part: $\xi=\xi\p+-\xi\p-$. Clearly, we have that $\xi\p\pm\in L\p2(\Om,\Gscr_\infty,\Pbb)$ and $\xi\p\pm\ge0$ can be represented as in \eqref{eq:rep.bou.rv}. Using now the linearity of the stochastic integrals, we obtain that $\xi$ can be represented as in \eqref{eq:rep.bou.rv} as well. To conclude the proof, we now use that the Hilbert spaces $(\Hscr\p2(\Gbb),\|\cdot\|_2)$ and $(L\p2(\Om,\Gscr_\infty,\Pbb),\|\cdot\|_2)$ are isomorphic. Therefore, we have that every $S\in\Hscr\p2(\Gbb)$ can be represented as
\[
S_t=\Ebb[S_\infty|\Gscr_t]=Y_0+K\cdot Z\p c_t+W\ast(\mu\p{Z}-\nu\p{Z})_t,\quad t\geq0,
\]
where $K\in\Lrm\p2(Z\p c,\Gbb)$ and $W\in\Gscr\p2(\mu\p{Z},\Gbb)$. Because of Proposition \ref{prop:eq.wprp} , this means that the $\Rbb\p{d}\times\Rbb\p{\ell}$-valued semimartingale $Z=(X,Y)\p{tr}$ has the WRP with respect to $\Gbb$. The proof Step 6 is complete. The proof of the theorem is complete. 
\end{proof}
\begin{examples}\label{ex:exam}
We now discuss some cases in which the propagation of the WRP to the independently enlarged filtration $\Gbb$ immediately follows from Theorem \ref{thm:mar.rep.G}. Notice that in all the following examples we do not exclude that the semimartingales $X$ and $Y$ may have common jumps. Furthermore, $\Delta X$ and $\Delta Y$ may charge predictable times.
\\[.5cm]
\indent \textbf{(1)} Assume that the semimartingales $(X,\Fbb\p X)$ and $(Y,\Fbb\p Y)$ are independent \emph{step} processes (see \cite[Definition 11.55]{HWY92}) and $\Fbb:=\Fbb\p X\vee\Rscr\p X$, $\Hbb:=\Fbb\p Y\vee\Rscr\p Y$, where $\Rscr\p X$ and $\Rscr\p Y$ are some initial $\sig$-fields such that $\Fbb$ and $\Hbb$ are independent (i.e., such that $\{\Fscr_\infty\p X,\Rscr\p X\}$ and $\{\Fscr_\infty\p Y,\Rscr\p Y\}$ are independent). Then, \cite[Theorem 13.19]{HWY92} yields that $X$ has the WRP with respect to $\Fbb$ and $Y$ has the WRP with respect to $\Hbb$. Because of Theorem \ref{thm:mar.rep.G}, we deduce that the $\Rbb\p 2$-valued $\Gbb$-semimartingale $Z=(X,Y)\p{tr}$ has the WRP with respect to $\Gbb=\Fbb\vee\Hbb$.
\\[.5cm]
\indent \textbf{(2)}  We can generalize the case in {(1)} as follows: Let $X$ and $Y$ be as in {(1)}. Let $B$ and $W$ be two independent Brownian motions such that $B$ is independent of $X$ and $W$ is independent of $Y$. We define $R:=B+X$ and $S:=W+Y$. Assume furthermore for simplicity that $\Rscr\p X$ and $\Rscr\p Y$ in {(1)} are both trivial. By \cite[Corollary 2]{WG82}, $R$ has the WRP with respect to $\Fbb\p R$ and $S$ has the WRP with respect to $\Fbb\p S$. Assuming that $\Fbb\p R$ and $\Fbb\p S$ are independent, Theorem \ref{thm:mar.rep.G} implies that the $\Rbb\p 2$-valued $\Gbb$-semimartingale $Z=(R,S)\p{tr}$ has the WRP with respect to $\Gbb=\Fbb\p R\vee\Fbb\p S$.
\\[.5cm]
\indent \textbf{(3)} 
We take $Y$ and $\Hbb$ as in {(1)} but $(X,\Fbb\p X)$ is assumed to be an $\Rbb\p d$-valued semimartingale with conditionally independent increments with respect to $\Fbb=\Fbb\p X\vee\Rscr\p X$. Again we assume that $\Fbb$ and $\Hbb$ are independent. From \cite[Theorem III.4.34]{JS00} (i), $X$ has the WRP with respect to $\Fbb$. Hence, Theorem \ref{thm:mar.rep.G} yields that the $\Rbb\p d\times\Rbb$-valued $\Gbb$-semimartingale $Z=(X,Y)\p{tr}$ has the WRP with respect to $\Gbb=\Fbb\vee\Hbb$.
\\[.5cm]
\indent \textbf{(4)} Let $(X,\Fbb\p X)$ be an $\Rbb\p d$-valued semimartingale and let $(Y,\Fbb\p Y)$ be an $\Rbb\p\ell$-valued semimartingale. Let $\Rscr\p X$ and $\Rscr\p Y$ denote two initial $\sig$-fields. We  assume that $X$ has conditionally independent increments with respect to $\Fbb=\Fbb\p X\vee\Rscr\p X$ and $Y$ has conditionally independent increments with respect to $\Hbb:=\Fbb\p Y\vee\Rscr\p Y$.  As an immediate consequence of \cite[Theorem III.4.34]{JS00} and of Theorem \ref{thm:mar.rep.G}, if $\Fbb$ and $\Hbb$ are independent, we get  that $Z=(X,Y)\p{tr}$ possesses the WRP with respect to $\Gbb=\Fbb\p X\vee\Fbb\p Y$.  
\\[.5cm]
\indent \textbf{(5)} Combining Theorem \ref{thm:mar.rep.G} and \cite[Theorem 2]{WG82}, we can construct several new semimartingales possessing the WRP. Ideed, let $X\p{(1)}$ and $X\p{(2)}$ be real-valued semimartingales possessing the WRP with respect to $\Fbb\p 1$ and $\Fbb\p 2$, respectively. Assume that $\Fbb\p1$ and $\Fbb\p 2$ are independent and set $\Fbb=\Fbb\p1\vee\Fbb\p 2$. Assume furthermore that the set of the $\Fbb$-predictable jump-times of $X\p{(1)}$ and $X\p{(2)}$ are disjoint and that the second and the third $\Fbb$-semimartingale characteristics of $X\p{(1)}$ and $X\p{(2)}$ are \emph{mutually singular} on $\Pscr(\Fbb)\otimes\Bscr(\Rbb)$ (see \cite[Theorem 2]{WG82}). Then, by \cite[Theorem 2]{WG82}, the semimartingale $X=X\p{(1)}+X\p{(2)}$ possess the WRP with respect to $\Fbb$ (see also \cite[Corollary 1 and Corollary 2]{WG82}). We consider the semimartingales $Y\p{(1)}$ and $Y\p{(2)}$ with respect to $\Hbb\p 1$ and $\Hbb\p 2$. We set $\Hbb:=\Hbb\p1\vee\Hbb\p2$ and make on $Y\p{(1)}$ and $Y\p{(2)}$ with respect to $\Hbb\p 1$ and $\Hbb\p 2$, respectively, similar assumptions as for $X\p{(1)}$ and $X\p{(2)}$. Then, the semimartingale $Y=Y\p{(1)}+Y\p{(2)}$ has the WRP with respect to $\Hbb$. If we now assume that $\Fbb$ and $\Hbb$ are independent, we get by Theorem \ref{thm:mar.rep.G} that the $\Rbb\p2$-valued semimartingale $Z=(X,Y)\p{tr}$ has WRP with respect to $\Gbb:=\Fbb\vee\Hbb$.
\\[.5cm]
\indent \textbf{(6)} The counterexample constructed in \cite{WG82} after Corollary 2 therein can be also handled with the help of Theorem \ref{thm:mar.rep.G}. We now use the notation of \cite{WG82}: Let $X\p{(1)}$ and $X\p{(2)}$ denote the processes introduced in  \cite{WG82} after Corollary 2. Then, Wang and Gang showed that the semimartingale $X\p{(1)}+X\p{(2)}$ does not possess the WRP with respect to the filtration $\Fbb:=\Fbb\p1\vee\Fbb\p2$. However, by  Theorem \ref{thm:mar.rep.G}, we see that the $\Rbb\p2$-valued semimartingale $Z=(X\p{(1)},X\p{(2)})\p{tr}$ possesses the WRP with respect to $\Fbb$.
\end{examples}
\subsection{Propagation of the Predictable Representation Property}\label{subsec:rep.prp}
In this subsection we investigate the propagation of the \emph{predictable representation property} to the independently enlarged filtration.
To begin with, we state the following definition of the predictable representation property.
\begin{definition}\label{def:PRP}
Let $X=(X\p1,\ldots,X\p d)\p{tr}$ be such that $X\p i\in\Hloc\p 1(\Fbb)$, $i=1,\ldots,d$. We say that the multidimensional local martingale $X$ has the \emph{predictable representation property} (from now on PRP) with respect to $\Fbb$ if, for every $Y\in\Mloc(\Fbb)$, there exists $K\in\Lrm\p1_\mathrm{loc}(X,\Fbb)$, such that $Y=Y_0+K\cdot X$ holds.
\end{definition}

In the next proposition, we state the relation between the PRP and the WRP for \emph{multidimensional} local martingales.  We stress that at this point we cannot directly use \cite[Theorem 13.14]{HWY92} because that result is only formulated for $\Rbb$-valued (and not for multidimensional) local martingales. This is a deep difference (although the formulation can be given in a notationally similar way) because, in the proof of Proposition \ref{prop:WRP.PRP} below, the stochastic integral for multidimensional local martingales is needed instead of the usual stochastic integral.  
\begin{proposition}\label{prop:WRP.PRP}
Let $X=(X\p1,\ldots,X\p d)\p{tr}$ be such that $X\p i\in\Hscr\p 1_\mathrm{loc}(\Fbb)$, $i=1,\ldots,d$. Assume that $X$ possess the \emph{PRP} with respect to $\Fbb$. Then, $X$ possesses the \emph{WRP} with respect to $\Fbb$.
\end{proposition}
We postpone the proof of Proposition \ref{prop:WRP.PRP} to the appendix. We stress that its converse is, in general, not true: If for example $X$ is an $\Rbb$-valued homogeneous L\'evy process and a martingale, then $X$ possesses the WRP with respect to $\Fbb\p X$ but it possesses the PRP with respect to $\Fbb\p X$ if and only if it is a Brownian motion or a compensated Poisson process (see, e.g., \cite[Corollary 13.54]{HWY92}).

\begin{corollary}[to Theorem \ref{thm:mar.rep.G}: Propagation of the PRP]\label{cor:srstr.rep}
Let $\Fbb$ and $\Hbb$ satisfy Assumption \ref{ass:indep}. Let $(X,\Fbb)$ be an $\Rbb\p d$-valued local martingale possessing the \emph{PRP} with respect to $\Fbb$ and let $(Y,\Hbb)$ be a semimartingale. Then, the $\Rbb\p{d}\times\Rbb\p{\ell}$-valued $\Gbb$-semimartingale $Z=(X,Y)\p{tr}$ possesses the \emph{WRP} with respect to $\Gbb$ if one of the following two conditions is satisfied:

\textnormal{(i)} $Y$ has the \emph{WRP} with respect to $\Hbb$.

\textnormal{(ii)} $Y$ is a $\Hbb$-local martingale possessing the \emph{PRP} with respect to $\Hbb$.
\end{corollary}
\begin{proof}
The statement is an immediate consequence of Proposition \ref{prop:WRP.PRP} and of Theorem \ref{thm:mar.rep.G}.
\end{proof}
\begin{remark}\label{rem:strong-strong}
We recall that, if  $X$ has the PRP with respect to $\Fbb=\Fbb\p X$, $Y$ has the PRP with respect to $\Hbb=\Fbb\p Y$ and if $\Fbb\p X$ and $\Fbb\p Y$ are independent, Calzolari and Torti showed in \cite{CT19} (under some additional conditions as, in particular, the triviality of $\Fscr\p X_0$ and $\Fscr\p Y_0$  and the locally-square integrability of $X$ and $Y$) that every $S\in\Hscr\p2(\Gbb)$ can also be represented as
\begin{equation}\label{eq:rep.st}
S=Y_0+K\cdot X+H\cdot Y+V\cdot[X,Y],
\end{equation}
for $K\in\Lrm\p2(X,\Gbb)$, $H\in\Lrm\p2(M,\Gbb)$, $V\in\Lrm\p2([X,M],\Gbb)$. Corollary \ref{cor:srstr.rep} (iii) shows that the representation in \eqref{eq:rep.st} can be formulated in terms of WRP with respect to $Z=(X,Y)\p{tr}$.
\end{remark}
\section{Applications}\label{sec:apl}
In this part we discuss two consequences of Theorem \ref{thm:mar.rep.G}. First, in Subsection \ref{subs:seq.en} we show the propagation of the WRP to an iteratively independent enlarged filtration. In Subsection \ref{subs:jeq}, we show the propagation of the WRP to the progressively enlargement by a random time $\tau$ satisfying Jacdod's equivalence hypothesis. 
\subsection{The Iterated Enlargement}\label{subs:seq.en} 
 
We recall that the filtrations $\Fbb\p 1,\ldots,\Fbb\p n$ are called \emph{jointly} independent, if $\{\Fscr\p 1_\infty,\ldots,\Fscr\p n_\infty\}$ is an independent family of $\sig$-algebras.
\begin{theorem}\label{thm:prop.seq.en}
Let $(X\p i,\Fbb\p i)$ be an $\Rbb\p{d_i}$-valued semimartingale, $d_i\in\Nbb$, and let $(B\p i,C\p i,\nu\p i)$ denote the $\Fbb\p i$-predictable characteristics of $X\p i$, $i=1,\ldots,n$. We denote $\Gbb:=\bigvee_{i=1}\p n\Fbb\p i$ and we assume that $\Fbb\p 1,\ldots,\Fbb\p n$ are jointly independent. We then have:

\textnormal{(i)} The filtration $\Gbb$ is right-continuous.

\textnormal{(ii)} $X\p i$ is a $\Gbb$-semimartingale and its $\Gbb$-predictable characteristics are $(B\p i,C\p i,\nu\p i)$, $i=1,\ldots,n$.

\textnormal{(iii)} If $X\p i$ possesses the \emph{WRP} with respect to $\Fbb\p i$, $i=1,\ldots,n$, then the $\Rbb\p{d_1}\times\cdots\times\Rbb\p{d_n}$-valued semimartingale $Z=(X\p 1,\ldots,X\p n)\p{tr}$ possesses the \emph{WRP} with respect to $\Gbb$.
\end{theorem}
\begin{proof}
We show the result by induction, as a direct consequence of Theorem \ref{thm:mar.rep.G}. To this aim, we observe that the joint independence of $\Fbb\p 1,\ldots,\Fbb\p n$ is equivalent to the joint independence of $\Fbb\p 1,\ldots,\Fbb\p{n-1}$ together with the independence of the family $\{\Fbb\p 1,\ldots,\Fbb\p{n-1}\}$ of $\Fbb\p n$. Now, we start with the inductive argumet. If $n=1$, there is nothing to show. We assume that (i), (ii) and (iii) hold for $n=m-1$. We are going to verify them for $n=m$. By the induction hypothesis and Lemma \ref{lem:What}, we immediately obtain (i). Analogously, from Proposition \ref{prop:G.char}, we deduce (ii). Let us now define $X:=(X\p 1,\ldots,X\p{m-1})\p{tr}$ and $Y=X\p m$. Since $\Gbb=\big(\bigvee_{i=1}\p{m-1}\Fbb\p i\big)\vee\Fbb\p m$, by the induction hypothesis and Theorem \ref{thm:mar.rep.G}, we obtain that $Z$ has the WRP with respect to $\Gbb$, which is (iii). The proof of the theorem is complete.
\end{proof}
\begin{corollary}
Let $(X\p i,\Fbb\p i)$ be an $\Rbb\p{d_i}$-valued local martingale (i.e., $X\p{ij}\in\Mloc(\Fbb\p i)$, $j=1,\ldots,d_i$, $i=1,\ldots,n$) such that $X\p i$ possesses the \emph{PRP} with respect to $\Fbb\p i$, $i=1,\ldots,n$. We denote $\Gbb:=\bigvee_{i=1}\p n\Fbb\p i$ and we assume that $\Fbb\p 1,\ldots,\Fbb\p n$ are jointly independent. Then, $(X\p i,\Gbb)$ are local martingales, $i=1,\ldots,n$, and the $\Rbb\p{d_1}\times\cdots\times\Rbb\p{d_n}$-valued $\Gbb$-local martingale $Z=(X\p 1,\ldots,X\p n)\p{tr}$ possesses the \emph{WRP} with respect to $\Gbb$.
\end{corollary}
\begin{proof}
Because of the independence, we obtain by induction and Lemma \ref{lem:prop}, that $(X\p i,\Gbb)$ are local martingales, $i=1,\ldots,n$. From Proposition \ref{prop:WRP.PRP} and Theorem \ref{thm:prop.seq.en} (iii), we immediately get the WRP of $Z$ with respect to $\Gbb$. The proof of the corollary is complete. 
\end{proof}

\subsection{Jacod's Equivalence Hypothesis}\label{subs:jeq}
Let $(X,\Fbb)$ be an $\Rbb\p d$-valued semimartingale and let $(B\p X,C\p X,\nu\p X)$ be the $\Fbb$-predictable characteristics of $X$. We assume that $\Fbb$ satisfies the usual conditions. Let $\tau:\Om\longrightarrow[0,+\infty]$ be a random time. We stress that $\tau$ is a random variable but it is not necessarily an $\Fbb$-stopping time. We denote by $H=1_{[\tau,+\infty)}$ the default process associated with $\tau$ and by $\Hbb$ the smallest filtration satisfying the usual conditions such that $H$ is $\Hbb$-adapted. We stress that, being a point process, $H$ possesses the WRP with respect to $\Hbb$. In this part we do not assume that $\tau$ and $\Fbb$ (i.e., that $\Hbb$ and $\Fbb$) are independent. We rather work under the following assumption (see, e.g., \cite[Definition 4.13]{AJ17}):
\begin{assumption}[Jacod's equivalence hypothesis]\label{ass:ja.eq} Let $F_\tau$ denote the law of $\tau$. The regular conditional distribution of $\tau$ given $\Fscr_t$ is equivalent to the distribution of $\tau$, that is, if $P_t(\cdot,A)$ denotes a version of $\Pbb[\tau\in A|\Fscr_t]$, $A\in\Bscr([0,+\infty])$, we have:
\[
P_t(\rmd u)\sim F_\tau(\rmd u),\quad\textnormal{a.s.}, 
\] 
where the symbol $\sim$ denotes the equivalence of the two measures. 
\end{assumption}
We denote by $\Gbb$ the smallest filtration containing $\Fbb$ and $\Hbb$, that is, $\Gbb=\Fbb\vee\Hbb$ is the progressive enlargement of $\Fbb$ by $\tau$. We notice that $\Gbb$ obviously coincides with the smallest filtration containing $\Fbb$ and such that $\tau$ is a $\Gbb$-stopping time.

In Assumption \ref{ass:ja.eq} we do not require the continuity of $F_\tau$. Therefore,  if $\tau$ satisfies Assumption \ref{ass:ja.eq}, then $\tau$ need not avoid $\Fbb$-stopping times, i.e., it can happen that $\Pbb[\tau=\sig<+\infty]>0$ for some $\Fbb$-stopping time $\sig$. If we require the continuity of $F_\tau$, we then have $\Pbb[\tau=\sig<+\infty]=0$, for every $\Fbb$-stopping time $\sig$ (see \cite[Corollary 2.2]{EJJ10}).

The propagation of the PRP for an $\Rbb$-valued local martingale $(X,\Fbb)$ to the progressive enlargement $\Gbb$ of $\Fbb$ by a random time $\tau$ satisfying Jacod's equivalence hypothesis, has been investigated in \cite{CJZ13} under the additional condition that $F_\tau$ is continuous. As a consequence of Theorem \ref{thm:mar.rep.G}, we are now going to show that under Assumption \ref{ass:ja.eq} the WRP propagates to $\Gbb$, also without the further continuity assumption on $F_\tau$. 

It is well-known that under Assumption \ref{ass:ja.eq} the following properties hold:
\begin{itemize}
\item $\Fbb$-semimartingales remain $\Gbb$-semimartingale (see \cite[Theorem 5.30]{AJ17}).  
\item There exists a version of the density of $\Pbb[\tau\in\rmd u|\Fscr_t]$ with respect to $F_\tau(\rmd u)$, denoted by $p_t(\om,u)$, such that $\Rbb_+\times\Om\times\Rbb_+\mapsto p_t(\om,u)$ is a strictly positive $\wt\Oscr(\Fbb)$-measurable function and, for every fixed $u\in\Rbb_+$, the process $(t,\om)\mapsto p_t(\om,u)$ is an $\Fbb$-martingale  (see \cite[Theorem 4.17]{AJ17} or \cite[Lemma 2.2]{A00}).
\item The process $L:=p_0(\tau)/p_\cdot(\tau)$ is a strictly positive $\Gbb\p\tau$-martingale, $\Gbb\p\tau$ denoting the \emph{initial enlargement} of $\Fbb$ by $\tau$, such that $L_0=1$ (see \cite[Theorem 4.37]{AJ17}).
\end{itemize} 
Using the process $L$ defined above as a density, according to \cite[Theorem 4.37]{AJ17}, for every arbitrary but fixed \emph{deterministic} time $T>0$, we can define the probability measure $\Qbb$ on $\Gscr_T\p\tau$ by setting
$\rmd\Qbb\big|_{\Gscr_T\p\tau}:=L_T\rmd\Pbb\big|_{\Gscr_T\p\tau}$ which has the following properties:
\begin{enumerate}
\item[(P1)] $\Qbb\big|_{\Gscr_T\p\tau}\sim\Pbb\big|_{\Gscr_T\p\tau}$.
\item[(P2)] $\Qbb\big|_{\Fscr_T}=\Pbb\big|_{\Fscr_T}$ and $\Qbb\big|_{\sigma(\tau)}=\Pbb\big|_{\sigma(\tau)}$.
\item[(P3)] $\Fscr_T$ and $\sig(\tau)$ are conditionally independent given $\Fscr_0$.
\end{enumerate}
We stress that (P1) only holds on $\Gscr\p\tau_T$ and not on $\Gscr\p\tau_\infty$ because, in general, $L$ is not a uniformly integrable martingale (see \cite[Remark 4.38]{AJ17}).

Since the inclusions $\Fscr_t\subseteq\Gscr_t\subseteq\Gscr_t\p\tau$, $t\in[0,T]$, hold, $\Fscr_T$ containing all the $\Pbb$-null sets of $\Fscr$, we also get $\Qbb\big|_{\Gscr_T}\sim\Pbb\big|_{\Gscr_T}$, $\Qbb\big|_{\Hscr_T}=\Pbb\big|_{\Hscr_T}$. Furthermore,  $\Fscr_T$ and $\Hscr_T\subseteq\sig(\tau)$ are conditionally independent given $\Fscr_0$.
Notice that, from \cite[Lemma 2.10]{CJZ13} and the comment after the proof therein, we obtain $\rmd\Qbb\big|_{\Gscr_T}=\ell_T\rmd\Pbb\big|_{\Gscr_T}$, where $\ell$ is a $\Gbb$-martingale.

For the remaining part of this section, $\Qbb$ will denote the equivalent probability measure described above.

Let $T>0$ be an arbitrary but fixed deterministic time. We denote $\Fbb\p T:=(\Fscr_{t\wedge T})_{t\ge0}=(\Fscr_t)_{t\in[0,T]}$ and  $\Hbb\p T:=(\Hscr_{t\wedge T})_{t\ge0}=(\Hscr_t)_{t\in[0,T]}$. 

\begin{lemma}\label{lem:righ.con}
Let $\Fscr_0$ be trivial. Then $\Gbb\p T:=\Fbb\p T\vee\Hbb\p T$ is right-continuous.
\end{lemma}
\begin{proof}
Because of (P3) and the triviality of $\Fscr_0$, we have that $\Fbb\p T$ and $\Hbb\p T$ are independent under $\Qbb$. Hence, we get the right-continuity of $\Gbb\p T$ by Lemma \ref{lem:prop}. The proof of the lemma is complete.
\end{proof}
\begin{theorem}\label{thm:wrp.eq}
Let $\tau$ satisfy Assumption \ref{ass:ja.eq} and let $(X,\Fbb)$ be an $\Rbb\p d$-valued semimartingale. Let us furthermore assume that $\Fscr_0$ is trivial and that $X$ possesses the \emph{WRP} with respect to $\Fbb$. Then, the $\Rbb\p d\times\Rbb_+$-valued $\Gbb$-semimartingale $Z=(X,H)\p{tr}$ possesses the \emph{WRP} with respect to $\Gbb\p T$.
\end{theorem}
\begin{proof}
Assumption \ref{ass:ja.eq} implies that $X$ is a $\Gbb$-semimartingale with respect to the probability measure $\Pbb$. Furthermore, since $\Qbb$ coincides with $\Pbb$ on $\Fscr_T$, we deduce that, under $\Qbb$, the $\Qbb$-semimartingale $X$ possesses the WRP with respect to $\Fbb\p T$. Furthermore, under $\Qbb$, $H$ possesses the WRP with respect to $\Hbb\p T$ (see \cite[Theorem III.4.37]{JS00}). Because of (P3) and the triviality of $\Fscr_0$, the filtrations $\Fbb\p  T$ and $\Hbb\p T$ are independent under $\Qbb$. We can therefore apply Theorem \ref{thm:mar.rep.G} under $\Qbb$ to obtain that the $\Gbb\p T$-semimartingale $Z$ possesses the WRP with respect to $\Gbb\p T$ under the measure $\Qbb$. So, by the equivalence of $\Qbb$ and $\Pbb$ on $\Gscr_T$, applying \cite[Theorem III.5.24]{JS00}, we obtain that $Z$ has the WRP with respect to $\Gbb\p T$ also under the original measure $\Pbb$. The proof of the theorem is complete.
\end{proof}

\appendix
\section{Technical Results and Proofs}\label{sec:tech.res}
\begin{lemma}\label{lem:multidim.stoc.int.itegrab} Let $q\geq1$. Let $X=(X\p1,\ldots,X\p d)\p{tr}$ be such that $X\p i\in\Hloc\p q(\Fbb)$, $i=1,\ldots,d$, and let $K\in\Lrm_\mathrm{loc}\p1(X)$. Then, $K\cdot X\in\Hscr\p q(\Fbb)$ if and only if $K\in\Lrm\p q(X)$.
\end{lemma}
\begin{proof}
Let $K\in\Lrm_\mathrm{loc}\p1(X)$. Then, it is known (see \cite[Eq.\ (4.57) and Remark 4.61 2)]{J79}) that the identity $[K\cdot X,K\cdot K]=(K\p{tr}aK)\cdot A$. The claim is therefore an immediate consequence of Burkholder--Davis--Gundy's inequality. The proof of the lemma is complete. 
\end{proof}

\begin{proof}[Proof of Proposition \ref{prop:WRP.PRP}]
Let $K\in\Lrm\p1(X)$. We preliminary observe that, from \cite[Theorem 4.68]{J79}, we have $K\in L\p1(X\p c)$ and $(K\cdot X)\p c=K\cdot X\p c$. Therefore, setting $X\p d:=X-X\p c-X_0$, we also get $K\in\Lrm\p1(X\p d)$ and $(K\cdot X)\p d=K\cdot X\p d$. Let now $\Hscr\p{1,d}(\Fbb)$ denote the space of purely-discontinuous martingales in $\Hscr\p{1}(\Fbb)$ (see \cite[Definition I.4.11b)]{JS00} for the definition of a purely discontinuous local martingale). We are going to show the inclusion $\Hscr\p{1,d}(\Fbb)\subseteq\{H\cdot X\p d,\ H\in L\p1(X\p d,\Fbb)\}$. So, let $M\in\Hscr\p{1,d}(\Fbb)$. Since $M_0=0$ by definition, from the PRP of $X$ with respect to $\Fbb$ and Lemma \ref{lem:multidim.stoc.int.itegrab}, we have $M=\cdot X$, where $H\in L\p1(X,\Fbb)$. Hence, by $M\p c=0$, we get $M=M\p d=(H\cdot X)\p d=H\cdot X\p d$ and $H\in L\p1(X\p d,\Fbb)$, that is, $M=H\cdot X\p d$, showing the claimed inclusion. This yields $\Hscr\p{1,d}(\Fbb)=\{W\ast(\mu\p X-\nu\p X),\ W\in\Gscr\p1(\mu\p X,\Fbb)\}$, by \cite[Theorem 4.80]{J79}. Therefore, if now $N\in\Hscr\p1(\Fbb)$ has a non-vanishing continuous part, from the PRP of $X$ with respect to $\Fbb$ and Lemma \ref{lem:multidim.stoc.int.itegrab}, we get 
\[N=N_0+H\cdot X=N_0+H\cdot X\p c+H\cdot X\p d=N_0+H\cdot X\p c+W\ast(\mu\p X-\nu\p X),\] where $ H\in L\p1(X\p c,\Fbb)$ and $W\in\Gscr\p1(\mu\p X,\Fbb)$. This means that $X$ has the WRP with respect to $\Fbb$. The proof of the proposition is complete.
\end{proof}

\begin{lemma}\label{lem:WRP.lim}
Let $(\xi\p n)_n\subseteq L\p2(\Om,\Fscr_\infty,\Pbb)$, $\xi\p n\longrightarrow\xi$ in $L\p2(\Om,\Fscr_\infty,\Pbb)$ as $n\rightarrow+\infty$. If 
\[
\xi\p n=\Ebb[\xi\p n|\Fscr_0]+K\p n\cdot X\p c_\infty+W\p n\ast(\mu\p X-\nu\p X)_\infty,\quad K\p n\in\Lrm\p2(X\p c,\Fbb),\quad W\p n\in\Gscr\p2(\mu\p X,\Fbb),
\]
then there exist $K\in\Lrm\p2(X\p c)$ and $W\in\Gscr\p2(\mu\p X)$ such that
\[
\xi=\Ebb[\xi|\Fscr_0]+K\cdot X\p c_\infty+W\ast(\mu\p X-\nu\p X)_\infty.
\]
\end{lemma}
\begin{proof}
We have
\[
\Ebb[(\xi\p n-\xi\p m)\p2]=\Ebb[\Ebb[\xi\p n-\xi\p m|\Fscr_0]\p2]+\Ebb[((K\p n-K\p m)\cdot X\p c_\infty)\p2]+\Ebb[((W\p n-W\p m)\ast(\mu\p X-\nu\p X)_\infty)\p2].
\]
So, by the isometry, $(K\p n)_{n\geq1}$ is a Cauchy sequence in $(\Lrm\p2(X\p c),\|\cdot\|_{\Lrm\p2(X\p c)})$ and $(W\p n)_{n\geq1}$ is a Cauchy sequence in $(\Gscr\p2(\mu\p X),\|\cdot\|_{\Gscr\p2(\mu\p X)})$. Therefore, we find $K\in \Lrm\p2(X\p c)$ and $W\in\Gscr\p2(\mu\p X)$ such that $K\p n\longrightarrow K$ in $\Lrm\p2(X\p c)$ and $W\p n\longrightarrow W$ in $\Gscr\p2(\mu\p X)$ as $n\rightarrow+\infty$, respectively. Considering now the stochastic integrals $K\cdot X\p c$ and $W\ast(\mu\p X-\nu\p X)$, we see that $\xi=\Ebb[\xi|\Fscr_0]+K\cdot X\p c_\infty+W\ast(\mu\p X-\nu\p X)_\infty$ holds. The proof of the lemma is complete.
\end{proof}
We define the functions $g_1,g_2:\Rbb\p {d}\times\Rbb\p {\ell}\longrightarrow\Rbb$ by $g_1(x,y):=1_{\{x\neq0\}}$ and $g_2(x,y):=1_{\{y\neq0\}}$, for every $(x,y)\in\Rbb\p {d}\times\Rbb\p {\ell}$.
\begin{lemma}\label{lem:What}
Let $\Dbb=(\Dscr_t)_{t\geq0}$ be a right-continuous filtration. Let $(X,\Dbb)$ be an $\Rbb\p {d}$-valued semimartingale and let $(Y,\Dbb)$ be an $\Rbb\p\ell$-valued semimartingale. Let us consider the $\Rbb\p{d}\times\Rbb\p{\ell}$-valued $\Dbb$-semimartingale $Z=(X,Y)\p{tr}$. Then, for every $W\in\Gloc\p1(\mu\p {X},\Dbb)$ and for every $V\in\Gloc\p1(\mu\p {Y},\Dbb)$ we have:

\textnormal{(i)}  The inclusion $Wg_1,Vg_2\in\Gloc\p1(\mu\p{Z},\Dbb)$ holds.

\textnormal{(ii)} The identities $Wg_1\ast(\mu\p{Z}-\nu\p{Z})=W\ast(\mu\p{X}-\nu\p{X})$ and $Vg_2\ast(\mu\p{Z}-\nu\p{Z})=V\ast(\mu\p{Y}-\nu\p{Y})$ hold.
\end{lemma}
\begin{proof}
We only show (i) and (ii) for $W\in\Gloc\p1(\mu\p{X},\Dbb)$, the proof for $V\in\Gloc\p1(\mu\p{Y},\Dbb)$ being completely analogous. First, we verify (i). Since $W\in\Gloc\p1(\mu\p {X},\Dbb)$, the predictable process $\wh W\p{X}$ is finite-valued and a version  of the $\Dbb$-predictable projection of the process $W(\cdot,\cdot,\Delta X)1_{\{\Delta X\neq0\}}$. 
Furthermore, we observe that the process $\wh{{Wg_1}}\p{Z}$ is always defined and $\Dbb$-predictable. Because of $\{\Delta X\neq0\}\subseteq\{\Delta Z\neq0\}$, we obviously have $1_{\{\Delta X\neq0\}}1_{\{\Delta Z\neq0\}}=1_{\{\Delta X\neq0\}}$. Hence, for every finite-valued $\Dbb$-predictable stopping time $T$, we have
 \[
\wh{W }\p{X}_T=\Ebb[W(T,\Delta X_T)1_{\{\Delta X_T\neq0\}}|\Dscr_{T-}]=\Ebb[W(T,\Delta X_T)g_1(\Delta X_T,\Delta Y_T)1_{\{\Delta Z_T\neq0\}}|\Dscr_{T-}]=\wh{{Wg_1}}_T\p{Z}\quad \textnormal{a.s.}
\]
Therefore, as a consequence of the predictable section theorem (see, e.g., \cite[Theorem IV.86]{DM82}), we deduce that $\wh W\p {X}$ and $\wh{Wg_1}\p{Z}$ are indistinguishable. Using the definition of $\wt{Wg_1}\p{Z}$, we now get
\begin{equation}\label{eq:idWt}
\begin{split}
\wt{Wg_1}\p{Z}(\om,t)&=W(\om,t,\Delta X_t(\om))g_1(\Delta X_t(\om),\Delta Y_t(\om))1_{\{\Delta Z_t(\om)\neq0\}}-\wh{Wg_1}\p{Z}(\om,t)\\&=\wt W\p{X}(\om,t),\quad \textnormal{ for every }\ t\geq0\ \textnormal{ a.s.}
\end{split}
\end{equation}
From \eqref{eq:idWt}, we deduce that $Wg_1\in\Gloc\p1(\mu\p{Z},\Dbb)$ holds. The proof of (i) is complete. We now verify (ii). Because of \ref{eq:idWt}, we have
\[
\Delta Wg_1\ast(\mu\p Z-\nu\p Z)=\wt{Wg_1}\p{Z}=\wt W\p{X}=\Delta W\ast(\mu\p{X}-\nu\p{X}),
\]
where the equalities have to be understood in the sense of indistinguishability. Hence, $W\ast(\mu\p{X}-\nu\p{X})$ and $Wg_1\ast(\mu\p{Z}-\nu\p{Z})$ are purely discontinuous local martingales with indistinguishable jumps. By \cite[Corollary I.4.19]{JS00}, this shows (ii). The proof of the lemma is complete.
\end{proof}
\begin{lemma}\label{lem:prop}
Let $\Fbb$ and $\Hbb$ be two independent right-continuous filtrations and set $\Gbb=\Fbb\vee\Hbb$. 

\textnormal{(i)} The filtration $\Gbb$ is right continuous.

\textnormal{(ii)} All $\Fbb$-local martingales (resp.\ $\Hbb$-local martingales) are $\Gbb$-local martingales.

\textnormal{(iii)} For every $\Fbb$-local martingale $M$ and for every $\Hbb$-local martingale $N$, the product $MN$ is a $\Gbb$-local martingale.
\end{lemma}
\begin{proof}
For (i) we refer to \cite[Theorem 1]{WG82}. To see (ii), we first verify the statement for uniformly integrable martingales. Let therefore $M$ be an $\Fbb$-martingale (the proof for $\Hbb$ is the same). By the independence assumption, we have that $\sig(M_t)\vee\Fscr_s$ is independent of $\Hscr_s$, for every $0\leq s\leq t$. Hence, we can compute
\[
\Ebb[M_t|\Gscr_s]=\Ebb[M_t|\Fscr_s\vee\Hscr_s]=\Ebb[M_t|\Fscr_s]=M_s.
\] 
If now $M$ is an $\Fbb$-local martingale and $(T_n)_{n\geq1}$ is a sequence of $\Fbb$-stopping times localizing $M$ to an $\Fbb$-martingale, by the previous step, we see that the stopped process $M\p{T_n}$ is a $\Gbb$-martingale, for every $n$. Hence, $(T_n)_{n\geq1}$ being also a sequence of $\Gbb$-stopping times, we deduce that $M$ is a $\Gbb$-local martingale. The proof of (ii) is complete. To see (iii) we again verify the statement first for martingales. Let $M$ be an $\Fbb$-martingale and let $N$ be a $\Hbb$-martingale. Since, by the independence assumption, we have that $\Ebb[|M_tN_t|]<+\infty$ and that $\sig(M_t)\vee\Fscr_s$ is independent of $\Hscr_t$, for every $0\leq s\leq t$, we can compute
\[
\Ebb[M_tN_t|\Gscr_s]=\Ebb[\Ebb[M_t|\Fscr_s\vee\Hscr_t]N_t|\Gscr_s]=M_s\Ebb[N_t|\Fscr_s\vee\Hscr_s]=M_sN_s,
\]
where, in the last identity, we used that $\sig(N_t)\vee\Hscr_s$ is independent of $\Fscr_s$. Let us now assume that $M$ is an $\Fbb$-local martingale and that $N$ is a $\Hbb$-local martingale. Let $(T_n)_{n\geq1}$ be a sequence of $\Fbb$-stopping times localizing $M$ to $\Hscr\p1(\Fbb)$ and let $(S_n)_{n\geq1}$ be a sequence of $\Hbb$-stopping times localizing $N$ to $\Hscr\p1(\Hbb)$ (we observe that we can always find such sequences of stopping times because of \cite[Lemma 2.38]{J79}). By the previous step, we know that $M\p{T_n}N\p{S_n}$ is a $\Gbb$-martingale. Furthermore, we have $M\p{T_n}N\p{S_n}\in\Hscr\p1(\Gbb)$, since $\Ebb[\sup_{t\geq0}|M\p{T_n}_tN\p{S_n}_t|]=\Ebb[\sup_{t\geq0}|M\p{T_n}_t|]\Ebb[\sup_{t\geq0}|N\p{S_n}_t|]<+\infty$ holds, by the independence assumption. We now define $R_n:=T_n\wedge S_n$, $n\geq1$. Then, $R_n\uparrow+\infty$, for $n\rightarrow+\infty$. Using Doob's stopping theorem, we deduce that $(MN)\p{R_n}=\big((M\p{T_n}N\p{S_n})\p{S_n}\big)\p{T_n}$ belongs to $\Hscr\p1(\Gbb)$, for every $n\geq1$. Hence, $MN$ is a $\Gbb$-local martingale. The proof is complete. 
\end{proof}

\end{document}